\documentclass[12pt]{article}
\usepackage{amssymb,latexsym, amsmath, amsthm, amsfonts, amssymb, enumerate,fullpage}
\usepackage[UTF8, scheme = plain]{ctex}
\usepackage[usenames,dvipsnames,svgnames,table]{xcolor}
\usepackage{color}
\usepackage{verbatim}
\usepackage{url}

\def\N{\mathbb{N}}
\def\R{\mathbb{R}}
\def\H{\mathbb{H}}

\def\Ric{{\rm Ric}}

\def\A{\mathcal{A}}
\def\B{\mathcal{B}}
\def\r{\rho}

\def\RR{\mathcal{R}}
\def\M{\mathcal{M}}
\def\NN{\mathcal{N}}
\def\W{\mathcal{W}}
\def\T{\mathcal{T}}
\def\HH{\mathcal{H}}

\def\e{\text{\rm{e}}}
\def\MM{\mathrm{M}}
\def\E{\mathcal{E}}
\def\K{\mathcal{K}}

\def\d{\text{\rm{d}}}
\def\u{u}
\def\v{v}
\def\x{x}
\def\y{y}
\def\X{\mathrm{X}}
\def\Y{\mathrm{Y}}

\def\xx{\mathrm{x}}
\def\yy{\mathrm{y}}
\def\zz{\mathrm{z}}
\def\<{\langle} \def\>{\rangle}

\DeclareMathOperator{\loc}{loc}
\DeclareMathOperator{\diam}{diam}

\newtheorem{theo}{\textsc{Theorem}}
\newtheorem{lem}[theo]{\textsc{Lemma}}
\newtheorem{propo}[theo]{\textsc{Proposition}}
\newtheorem{cor}[theo]{\textsc{Corollary}}
\newtheorem{remark}{Remark}

\title{Weak type $(1,  1)$ of Riesz transform on some direct product manifolds with exponential volume growth}
\author{Hong-Quan Li and Jie-Xiang Zhu}
\date{%\today
}
\begin{document}

\renewcommand{\theequation}{\thesection.\arabic{equation}}
\setcounter{equation}{0} \maketitle

\vspace{-1.0cm}

\bigskip

{\bf Abstract.} In this paper we are concerned with the Riesz transform on the direct product manifold $\H^n \times M$, where $\H^n$ is the $n$-dimensional real hyperbolic space and $M$ is a connected complete non-compact Riemannian manifold satisfying the volume doubling property and generalized Gaussian or sub-Gaussian upper estimates for the heat kernel. We establish its weak type $(1,1)$ property. In addition, we obtain the weak type $(1, 1)$ of the heat maximal operator in the same setting. Our arguments also work for a large class of direct product manifolds with exponential volume growth. Particularly, we provide a simpler proof of weak type $(1,1)$ boundedness of some operators considered in the work of Li, Sj\"ogren and Wu \cite{LSW16}.

\medskip


\medskip


\medskip

\renewcommand{\theequation}{\thesection.\arabic{equation}}
\section{Introduction and main results}
\setcounter{equation}{0}
\medskip

The weak $(1, 1)$ and $L^p$ boundedness of Riesz transform are one of the classic topics in harmonic analysis. In the last several decades, they have been investigated extensively in the setting of Riemannian manifolds. To formulate these problems, we first list the following terminologies and notations that will be used throughout the whole paper. Let $(M, g)$ be a connected complete non-compact Riemannian manifold. Let us denote by $\d \mu_M$ the Riemannian measure and by $d_M$ the geodesic distance. Denote by $B(x, R)$ the open ball of center $x$ and of radius $R > 0$. Set $V_M(x, R) = \mu_M(B(x, R))$. Let $\nabla_M$ be the Riemannian gradient, $\Delta_M$ be the Laplace-Beltrami operator on $M$. Denote by $(\e^{t\Delta_M})_{t>0}$ the heat semigroup associated with $\Delta_M$ and $p_t^M(x, y)$ the heat kernel. By spectral analysis,
$$(-\Delta_M)^{-\frac{1}{2}} = \frac{1}{\sqrt{\pi}} \int_{0}^{\infty} \e^{t\Delta_M} \, \frac{\d t}{\sqrt{t}}.$$
Then the Riesz transform $\RR$ on $M$ is defined by $\nabla_M (-\Delta_M)^{-\frac{1}{2}}$.  In the sequel, we use the notation $A \lesssim B$ if there exists a universal constant $C > 0$ such that $A \leq C B$, where the constant $C$ may only depend on the properties of the underlying manifolds. We write $A \sim B$, if $A \lesssim B$ and $B \lesssim A$. We also use $A \lesssim_\varepsilon B$ to denote $A \le C_\varepsilon B$, where the constant $C_\varepsilon$ depends on the parameter $\varepsilon$. In the present work, we shall consider the weak type $(1, 1)$ property of the Riesz transform on $M$, namely
\begin{align} \label{W11}
\mu_M \left( \left\{ x ; \, |\nabla_M (-\Delta_M)^{-\frac12} f(x)| > \lambda \right\} \right) \, \lesssim \, \frac{\| f \|_1}{\lambda}, \quad \forall \, \lambda > 0, \, f \in C_0^\infty(M);
\end{align}
and its $L^p$ boundedness
properties, namely, for which $1 < p < +\infty$,
\begin{align} \label{lpe}
\|\nabla_M (-\Delta_M)^{-\frac{1}{2}}f \|_p  \, \lesssim \, \| f \|_p, \quad \forall \, f \in C_0^{\infty}(M).
\end{align}
Since the case $p = 2$ is a direct result of integration by parts, then \eqref{lpe} for $1< p <2$ follows from \eqref{W11} by interpolation. The classical singular integral operator theory says that on the Euclidean space, \eqref{W11} and \eqref{lpe} are valid for all $1 < p < +\infty$ (see e.g. \cite[Chapter III]{S70}). In contrast with the classical result, for any given $p_0 > 2$, there exist stochastically complete Riemannian manifolds, on which \eqref{lpe} is valid iff $1 < p < p_0$, and even iff $1 < p \le 2$ (see e.g. \cite{CD99, Li99} for more details). It is worthwhile to mention that there are various recent counterexamples in the literature. %Remark \ref{nR1}
The appendix below contains a brief discussion of \eqref{lpe} for $p > 2$ in our setting. In this paper we mainly focus on \eqref{W11}.

\medskip

\subsection{Volume doubling property and heat kernel upper estimates}

\medskip

For the further purpose, it is worthwhile discussing some basic assumptions on $M$ and $p_t^M$. We say that $M$ satisfies the volume doubling property if there exists $D > 0$, such that
$$V_M (x, 2R)  \leq  D \, V_M(x, R), \qquad \forall \, x \in M, \,  R > 0. \eqno(D) $$

Given $m \geq 2$. Suppose that there exists some $c \in (0, 1)$, for all $x, y \in M$,
\begin{equation*}
p_t^M (x, y)  \, \leq \, c^{-1}
\begin {cases}
\frac{1}{V_M \left( x , \, \sqrt{t} \right)} \exp{\left( - c \, \frac{d_M(x, \, y)^2}{t} \right)},  & \text{if \, $0 < t < 1$;} \\
\\
\frac{1}{V_M \left( x, \, t^{\frac{1}{m}} \right)} \exp{\left(-c \, \left( \frac{d_M(x, \, y)^m}{t} \right)^{\frac{1}{m - 1}} \right)},  & \text{if \, $t \geq 1$.} \\
\end {cases} \eqno(UE_m)
\end{equation*}
The heat kernel estimate ($UE_2$) (respectively, ($UE_m$) for $m > 2$) is called Gaussian upper estimate (respectively, sub-Gaussian upper estimate with the exponent $m$).  ($UE_2$) is also known as Li-Yau upper estimate. If $M$ satisfies ($D$), then ($UE_2$) is equivalent to the relative Faber-Krahn inequality (see e.g.  \cite[Chapter 15]{G09}  for more details). While ($UE_m$) for $m > 2$ holds on some fractal-like manifolds (see \cite{CCFR16}).

Several different generalized forms of ($UE_m$) have been considered by other authors (see, e.g.,  \cite{DM99, BK03, BK04}). However, we need to introduce a new one in our context. As we know, ($UE_m$) is unfortunately not stable under products, i.e. provided that $M_1$ and $M_2$ are connected complete non-compact Riemannian manifolds satisfying  ($UE_{m_1}$) and ($UE_{m_2}$) with $m_1 \neq m_2$ respectively, then ($UE_m$) does not hold for any $m \geq 2$ on $M_1 \times M_2$ in general (see  \cite[\S 2]{LZ17}).
On the other hand, when we investigate Riesz transform on product manifolds, sometimes the geodesic distance
does not directly play a role so that we have to define other distances (or even quasi-distances) and work within the new framework (see \cite{LZ17}). Let $\widetilde{d}$ denote a quasi-distance in the sense of Coifman-Weiss on $M$ (see \cite{CW71}). Besides we assume that $\widetilde{d}$ induces the topology of $M$. Set
\begin{align*}
 \widetilde{B}(x, R) = \{y \in M; \, \widetilde{d}(x, y) < R \}, \qquad \widetilde{V}_M (x, R) = \mu_M \big(\widetilde{B}(x, R)\big), \qquad \forall \, x \in M, \, R > 0.
\end{align*}
We naturally hope our generalization is compatible with $\widetilde{d}$. These facts lead us to the following definition. Given $0 < \theta  \leq 1$, we say that $M$ satisfies the (so-called) generalized heat kernel upper estimates with the exponent $\theta$ w.r.t. the quasi-distance $\widetilde{d}$, if there exists a constant $c \in (0,1)$ such that
$$p_t^M (x, y) \leq c^{-1} \frac{1}{\widetilde{V}_M \left(x, \sqrt{t}\right)} \exp{\bigg(- c \, \bigg(\frac{\widetilde{d}(x, y)^2}{t}\bigg)^{\theta}\bigg)},  \quad \forall \, t > 0, \, x, y \in M. \eqno(GUE_{\theta}) $$
In the sequel, to avoid confusing of notations, we rewrite the volume doubling property of $(M, \widetilde{d}, \d\mu_M)$ as (for some constant $\widetilde{D} > 0$)
$$\widetilde{V}_M (x, 2R) \leq \widetilde{D} \, \widetilde{V}_M (x, R), \qquad \forall \, x \in M, \,  R > 0. \eqno(ND) $$
Note that ($ND$) implies the following stronger property: there exist positive constants $D'$ and $\upsilon$, such that
\begin{align}  \label{nd}
\widetilde{V}_M (x, R_1) \leq D' \bigg(1 + \frac{R_1}{R_2} \bigg)^\upsilon \widetilde{V}_M (x, R_2), \qquad \forall \, x \in M, \, R_1,R_2 > 0.
\end{align}

In addition, the following assumptions on weighted estimates of $\nabla_M p_t^M(x, y)$ \footnote{Hereafter, for any smooth function $u(x, y)$ defined on $M \times M$, we adopt the shorthand notation $\nabla_M u(x, y) := \nabla_{M, x}u(x, y)$.} will be requested. Given $1 \leq q \leq 2$ and $0 < \theta  \leq 1$, we say that $M$ satisfies weighted gradient integral estimates of the heat kernel with the exponent $(q, \theta)$ w.r.t. $\widetilde{d}$, if there exist constants $c, c' > 0$ such that for all $t > 0$, $y \in M$,
$$\int_M |\nabla_M p_t^M (x, y)|^q \exp{\bigg( c \bigg(\frac{\widetilde{d}(x, y)^2}{t}\bigg)^{\theta}\bigg)} \, \d\mu_M (x) \leq \frac{c'}{t^{\frac{q}{2}} \widetilde{V}_M \left(y, \sqrt{t}\right)^{q - 1}}.
\eqno(WGE_{q, \theta}) $$
It is worth noting that in the case $\widetilde{d} = d_M$, ($WGE_{2, 1}$) is deduced from ($ND$) and ($GUE_{1}$), see e.g. \cite{CD99}. It has recently been pointed out implicitly in \cite{CCFR16} that when $\widetilde{d} = \max\{d_M, d_M^{\frac{m}2} \}$ for some $m >2$, then ($WGE_{q, \frac{1}{m-1}}$) ($q \in (1, 2)$) is a corollary of  ($ND$) and ($GUE_{\frac{1}{m - 1}}$). But the above two conclusions are based on the fact that $|\nabla_M d_M| \leq 1$, which is obviously not valid for a general quasi-distance $\widetilde{d}$.
It should be reasonable to conjecture that, in general, ($WGE_{q, \theta}$) is not a direct consequence of the conjunction of ($ND$) and ($GUE_\theta$). That is the reason we list it here as an independent assumption. Section \ref{S} contains a detailed discussion of ($GUE_{\theta}$) and ($WGE_{q, \theta}$).

\medskip

\subsection{Some closely related known results about Riesz transform} \label{SS12}

\medskip

In the past several decades, a great deal of mathematical effort has been devoted to \eqref{W11} and \eqref{lpe} on various manifolds. Here we only mention a few early works \cite{S83, L85, B85, B87, L88, A92, Al92}. Meanwhile, people try to establish \eqref{W11}
under as few assumptions as possible, and the volume doubling property ($D$) is among the most essential ones. Since ($D$) implies the weak $(1, 1)$ estimate of Hardy-Littlewood maximal operator and the Calder\'on-Zygmund decomposition of $L^1$ functions, which are convenient tools in harmonic analysis. In this regard, a seminal result has been obtained by T. Coulhon and X. T. Duong in \cite{CD99}, using singular integral theory developed in \cite{DM99}. More precisely, they established \eqref{W11} under ($D$) and ($UE_2$). And recent progress can be found in \cite{CCFR16} and \cite{LZ17}. To be more precise, \eqref{W11} has been proved under the assumption that ($D$) and ($UE_m$) in \cite{CCFR16},  also under the weaker assumption that ($ND$), ($GUE_{\theta}$) and ($WGE_{1, \theta}$) in \cite{LZ17}.

However, there is a wide variety of manifolds not satisfying the volume doubling property. For example, manifolds with both exponential volume growth and a positive bottom of the spectrum, on which we will concentrate. Compared with doubling manifolds, the main issue in this setting is the lack of an adequate singular integral theory. Nevertheless \eqref{W11} has been obtained in the following few settings: non-compact symmetric spaces (see \cite{A92}), harmonic $AN$ groups (see \cite{ADY96}), the Laplacian with drift on Euclidean spaces (see \cite{LSW16}) and on real hyperbolic spaces in the sense of  \cite{LS17}, as well as the sub-Laplacian with drift on Heisenberg groups in the context of \cite{LS20}.

\medskip

\subsection{Real hyperbolic space} \label{real hyperbolic}

\medskip

The real hyperbolic space $\H^n$ of dimension $n \geq 2$ is a typical manifold with exponential volume growth and a positive bottom of the spectrum. We hereafter adopt the upper half-space representation (see, e.g., \cite[\S 5.7]{D89}): that is  $ \H^n = \R^{n - 1} \times \R^+$ endowed with the metric $g_{\H^n} = u_n^{-2} \sum_{i = 1}^n (\d u_i)^2$. The associated Riemannian measure is $\d \mu_{\H^n}(u) = u_n^{-n} \d u$, where $\d u$ is the Lebesgue measure. In this representation, the geodesic distance $r$ between $u = (u_1, \cdots, u_n), v = (v_1,\cdots, v_n) \in \R^{n - 1} \times \R^+$  can be computed as
$$r(u, v) = \text{arccosh }
\frac{ \sum_{i = 1}^{n-1}(u_i - v_i)^2+ u_n^2+ v_n^2}{2 u_n v_n}.$$

The gradient  of $f \in C^\infty(\H^n)$ is written as
$$ \nabla_{\H^n} f = \bigg( u_n \frac{\partial f}{\partial u_1}, \cdots, u_n \frac{\partial f}{\partial u_n} \bigg),$$
and the Laplace-Beltrami operator is given by
$$ \Delta_{\H^n} f = u_n^2 \sum_{i = 1}^n \frac{\partial^2 f}{\partial u_i^2} - (n-2) u_n \frac{\partial f}{\partial u_n}. $$

In Riemannian Geometry, it is well-known that the real hyperbolic space is homogeneous, the volume of geodesic ball of radius $R$ is independent of its center, and
\begin{equation} \label{ve}
 V_{\H^n} (u, R) \, \sim \,
 \begin {cases}
 R^n  & \text{if \, $0 < R \leq 1$,} \\
  \e^{(n-1)R}   & \text{if \, $R>1$.} \\
   \end {cases}
\end{equation}
Consequently, $\H^n$ satisfies the following local volume doubling property in the sense that for any fixed $R_0 > 0 $, there exists a positive constant $D_{R_0}$, such that
$$V_{\H^n} (u, 2R) \leq D_{R_0} V_{\H^n} (u, R), \qquad \forall \, u \in \H^n, \,  R \in (0, R_0). \eqno(ND_{\loc}) $$
As for the behavior at infinity, its volume growth is at most exponential in the sense that there exist positive constants $C$ and $c$ such that
$$V_{\H^n}(u, \theta R) \leq C \, \e^{c \, \theta} \, V_{\H^n}(u, R), \qquad \forall \, u \in \H^n, \, \theta > 1 , \, R \in (0, 1].  \eqno(VE)                                                                                      $$
Obviously the properties $(ND_{\loc})$ and $(VE)$ can be defined on general (quasi-)metric measure spaces.

For the Riesz transform on $\H^n$, \eqref{lpe} for $1 < p < +\infty$ and \eqref{W11} are known results.

\medskip

\subsection{Main results} \label{main}

\medskip

We shall consider a class of product Riemannian manifolds $\M$ in this paper. $\M$ is obtained by taking the direct product of $\H^n$ and a connected complete non-compact Riemannian manifold $M$ equipped with a quasi-distance $\widetilde{d}$. Besides we always assume $M$ satisfies ($ND$), ($GUE_{\theta}$) and ($WGE_{q, \theta}$) (with some $1 \leq q \leq 2$ and $0 < \theta \leq 1$) w.r.t. $\widetilde{d}$.

It is well-known that the heat kernel $p_t^{\M}$ on $\M$ is given by (see e.g. \cite{G09})
\begin{align} \label{hk}
p_t^\M (\X, \Y) = p_t^{\H^n}(\u, \v) \,  p_t^{M}(\x, \y), \quad  \forall \, t > 0, \, \, \X = (\u, \x), \, \Y = (\v, \y) \in \H^n \times M.
\end{align}
For our purpose, a quasi-distance $d$ on $\M$ is defined as follows:
\begin{align} \label{distance}
d(\X, \Y) = \max \{ r(\u, \v), \, \widetilde{d}(\x, \y)     \}, \quad \X = (\u, \x), \ \Y = (\v, \y) \in \H^n \times M.
\end{align}
Note that $d$ differs from the geodesic distance on $\M$ considered in differential geometry.
In the sequel, for simplicity, we will omit coordinates of points in distance functions if they are self-evident. The Riemannian measure $\d\mu_\M$ on $\M$ satisfies $\d \mu_\M = \d \mu_{\H^n} \otimes \d \mu_M$. Then ($ND_{\loc}$) of $\H^n$ and ($ND$) of $M$ imply that $(\M, d, \d \mu_\M)$ satisfies ($ND_{\loc}$). As for the gradient of $f \in C^1(\M)$, we have $\nabla_{\M} f  = (\nabla_{\H^n} f, \nabla_M f )$.

The motivation for this work comes from two aspects. Firstly, recalling all the existing weak type $(1,1)$ results on manifolds with exponential volume growth, whose bottom of the spectrum is positive (see \cite{A92, ADY96, LSW16, LS17, LS20}), approaches therein all require sharp pointwise estimates of spatial derivatives of the heat kernel, which are highly non-trivial.
Therefore it is natural to ask if the use of these estimates could be avoided when establishing \eqref{W11}.

Secondly, we are interested in the Riesz transform on direct product manifolds. We have showed in \cite{LZ17} that if $M$ is the direct product manifold $M_1 \times \cdots \times M_n$, where $M_i$ ($1 \leq i \leq n$) are connected complete non-compact Riemannian manifolds satisfying ($D$) and ($UE_{m_i}$) ($m_i \geq 2$), then \eqref{W11} and \eqref{lpe} for $1 < p < 2$ remain valid on $M$. Hence it is reasonable to ask whether the corresponding results hold in the setting of direct product manifolds with exponential volume growth, such as $\H^n \times M$, where $M$ is a connected complete non-compact Riemannian manifold satisfying ($ND$), ($GUE_{\theta}$) and $(WGE_{q, \theta})$. The aim of this paper is to give an affirmative answer to this question.

Now we can state our main result as follows:

\begin{theo} \label{nt1}
Let $M$ be a connected complete non-compact Riemannian manifold satisfying $(ND)$, $(GUE_{\theta})$ and $(WGE_{q, \theta})$ $($with $1 \leq q \leq 2$, $0 < \theta \leq 1$$)$ w.r.t. the quasi-distance $\widetilde{d}$. Moreover, suppose that for any given $R_0 > 0$, the ball $\widetilde{B}(x, R_0)$ equipped with the induced quasi-distance $\widetilde{d}$ and measure $\d \mu_M$ satisfies $(ND)$, with a doubling constant independent of its center. Then, the Riesz transform on $\H^n \times M$ is of weak type $(1, 1)$ $($so $L^p$ bounded for $1 < p \leq 2$$)$.
\end{theo}

\begin{remark} \label{r0}
The condition that $( \widetilde{B}(x, R_0), \widetilde{d}, \d \mu_M )$ satisfies $(ND)$ uniformly in $x \in M$ is a technical property that is only used in Section \ref{S2} for a localization argument. When $\widetilde{d} = d_M$ and $(ND_{\loc})$ holds, it is automatically fulfilled, as shown in \cite[Lemma 4.1]{ACDH03} (cf. also \cite[Lemme 4.3.1]{Li99}). In addition, it can be observed that this property remains stable under products, provided that direct product manifolds are equipped with a quasi-distance of the type \eqref{distance}. Therefore, under the assumptions of Theorem \ref{nt1}, $\H^n \times M$ shares the same property.
\end{remark}

The techniques developed here also apply to the study of the heat maximal operator on $\M$, which is defined by
\begin{align*}
\HH (f)(\X) := \sup_{t>0} |\e^{t \Delta_\M}f(\X)|, \quad \forall \, \X \in \M, \, f \in L^p (\M), \, 1 \leq p \leq + \infty.
\end{align*}
Since the heat semigroup $( \e^{t \Delta_\M} )_{t > 0}$ is a diffusion semigroup symmetric w.r.t. $\mu_\M$, the $L^p$ boundedness ($1 < p \le +\infty$) of $\HH$ can be found in \cite[Section 3.3]{S70(2)} or \cite[Lemma 1.6.2]{BGL14}. In Section \ref{S4}, we establish the weak type $(1, 1)$ property of $\HH$ as follows:

\begin{theo} \label{nt2}
Let $M$ be a connected complete non-compact Riemannian manifolds satisfying $(ND)$, $(GUE_{\theta})$ w.r.t. the quasi-distance $\widetilde{d}$ $($with $0 < \theta \leq 1$$)$. Then, the heat maximal operator on $\H^n \times M$ is  of weak type $(1, 1)$ $($so $L^p$ bounded for $1 < p \leq + \infty$$)$.
\end{theo}
For more results on the weak type $(1, 1)$ boundedness of $\HH$ in the context of manifolds with exponential volume growth, we refer the reader to \cite{A92, ADY96, CGGM91, LSW16, LS17, LS20}.

\medskip

Section \ref{S6} extends the above results to a large class of direct product manifolds with exponential volume growth. As a by-product, we generalize the main results in \cite{LSW16} and provide a much simpler proof.

\medskip

\renewcommand{\theequation}{\thesection.\arabic{equation}}
\section{Preliminaries}
\setcounter{equation}{0} \label{S}

\medskip

\subsection{Some basic facts about ($GUE_{\theta}$) and ($WGE_{q, \theta}$)}

To illustrate ($GUE_{\theta}$) and ($WGE_{q, \theta}$), recall first that if $M$ meets ($D$) and ($UE_2$), then ($GUE_1$) and ($WGE_{2, 1}$) are valid w.r.t. the geodesic distance (cf. e.g. \cite{CD99} or \cite{G95}).

Next, suppose that $M$ satisfies the volume doubling property and the sub-Gaussian heat kernel upper estimates with the exponent $m > 2$. Let
\begin{align*}
\widetilde{d}(x, y) = \max\{d_M(x, y), \, d_M(x, y)^{\frac{m}{2}}\}, \quad \forall \, x, y \in M.
\end{align*}
It is easy to check that $\widetilde{d}$ defines a quasi-distance on $M$ and that $(M, \widetilde{d}, \d \mu_M)$ satisfies ($ND$), ($GUE_{\frac{1}{m - 1}}$). It follows from \cite[Lemma 2.2]{CCFR16} that ($WGE_{q, \frac{1}{m - 1}}$) is available for all $1 < q < 2$ with regard to $\widetilde{d}$.

Fix a quasi-distance $\widetilde{d}$ on $M$. Notice that given $0 < \theta' < \theta$,  $s^{\theta'} \le 1 + s^\theta$ holds for all $s \geq 0$. Hence
\begin{align} \label{GUE}
\mbox{($GUE_{\theta}$)} \Longrightarrow \mbox{($GUE_{\theta'}$)}, \quad \mbox{($WGE_{q, \theta}$)} \Longrightarrow \mbox{($WGE_{q, \theta'}$)}.
\end{align}

Assume now that ($ND$) and ($GUE_{\theta}$) are valid. The standard method of decomposition in annuli implies that
\begin{align} \label{be1}
\int_M \exp{\Bigg( - \gamma  \, \bigg(\frac{\widetilde{d}(x, y)^2}{t}\bigg)^{\theta}\Bigg)} \, \d\mu_M(x) &= \int_{\widetilde{B}(y, \sqrt{t})} + \sum_{i = 0}^{+\infty} \int_{2^i \sqrt{t} \leq \widetilde{d}(y, \, x) < 2^{i + 1} \sqrt{t}} \nonumber\\
&\leq C(\gamma, \theta) \, \widetilde{V}_M (y, \sqrt{t} ), \quad \forall \,  t > 0, \,   y \in M,
\end{align}
where we have used \eqref{nd} in the last line. Combining this with H\"older inequality,
we obtain the following relation
\begin{align} \label{GUE2}
\mbox{($WGE_{q, \theta}$)} \Longrightarrow \mbox{($WGE_{q', \theta}$)}, \quad \forall \,  1 \leq q' < q.
\end{align}

To finish this section, we point out an invariance of ($GUE_\theta$) and ($WGE_{q, \theta}$) under direct product in the following sense:

\begin{propo} \label{prop1}
Suppose that  $M_i$ $(i  = 1, 2)$  are connected complete non-compact Riemannian manifolds satisfying $(ND)$, $(GUE_{\theta_i})$ and $(WGE_{q_i, \theta_i})$ $($with $1 \leq q_i \leq 2$, $0 < \theta_i \leq 1$$)$ w.r.t. the quasi-distance $\widetilde{d_i}$. Let $\theta = \min\{\theta_1, \theta_2\}$, $q = \min\{q_1, q_2\}$ and
\begin{align*}
\widetilde{d}(x, y) = \max_{i = 1, 2}\{\widetilde{d_i}(x_i, y_i)\}, \quad \forall \, x = (x_1, x_2), \, y = (y_1, y_2) \in M_1 \times M_2.
\end{align*}
Then on the direct product manifold $M = M_1 \times M_2$, we have $(ND)$, $(GUE_\theta)$ and $(WGE_{q, \theta})$ w.r.t. the quasi-distance $\widetilde{d}$.
\end{propo}

\begin{proof}
This proof is an easy task. It suffices to use properties of product manifolds and \eqref{be1}. So we omit it.
\end{proof}

\begin{remark}
The result remains valid if one of $M_i$ is a compact Riemannian manifold or a Lie group with polynomial growth.
\end{remark}

\medskip

\subsection{Some further properties of $\H^n$}

\medskip

In this short section, we collect basic properties of $\H^n$ which will be used throughout this work. We first present some estimates of the heat kernel and its gradient, which can be found in \cite{D89, DM88}:

The heat kernel $p_t^{\H^n}(u, v)$ depends only on $t$ and $r = r(u, v)$.
More precisely, it is known that for all $r \geq 0$ and $t > 0$,
\begin{align} \label{hk of hn}
p_t^{\H^n}(\u, \v) = p_t^{\H^n}(r) \sim t^{-\frac{n}{2}} \left( 1 + r+t \right) ^{\frac{n-3}{2}} \left( 1+r \right) \exp{ \bigg( -\frac{r^2}{4t}-\frac{n-1}{2}r - \frac{(n-1)^2}{4}t\bigg) }.
\end{align}
In particular, combining this with \eqref{ve}, $p_t^{\H^n}$ has the following small-time upper bound:
\begin{align} \label{small time gaussian}
p_t^{\H^n}(\u , \u)  \leq \frac{C}{V_{\H^n}(\u, \sqrt{t})}, \quad \forall \, \u \in \H^n, \, \, 0< t \leq 1,
\end{align}
and the following local estimate: given $r_0 > 0$, for all $\u, \v \in \H^n$ with $r = r(\u, \v) \leq r_0$,
\begin{align} \label{local estimate}
p_t^{\H^n}(\u, \v) \lesssim_{r_0} t^{-\frac{n}{2}} \exp \left(-\frac{r^2}{Ct}\right) , \quad \forall \, t > 0.
\end{align}

As for the spatial derivative of $p_t^{\H^n}$, it is known that for $\u, \v \in \H^n$ and $t > 0$,
\begin{align} \label{gr of hk}
\big|\nabla_{\H^n} p_t^{\H^n} (u, v) \big| \sim t^{-\frac{n+2}{2}} \left( 1+r+t \right) ^{\frac{n-1}{2}} r \exp{ \left( -\frac{r^2}{4t}-\frac{n-1}{2}r - \frac{(n-1)^2}{4}t\right) }.
\end{align}

We will make use of the following lemma, which asserts that on $\H^n$, a specific decay of the kernel toward infinity implies the weak type $(1, 1)$ property of the associated integral operator. For its proof, see \cite{S81}.

\begin{lem} \label{weak type on hn}
The integral operator defined on $\H^n$ with the kernel
$$S(u, v) = \e^{-(n - 1) \, r}$$
is of weak type $(1, 1)$.
\end{lem}

\medskip

\subsection{Weak type $(1, 1)$ of sub-linear operators on product measure spaces}

Now we present a simple criterion of weak type $(1, 1)$, which plays a role in our proof.

\begin{lem} \label{weak type 11}
Let $(X ,\mu)$ and $(Y, \nu)$ be measure spaces. Let $\A$ and $\B$ be sub-linear operators defined on $L^1(X, \mu)$ and $L^1(Y, \nu)$ respectively, which are of the forms
$$ \A(f_1)(x_1) = \int_{X}   a(x_1, y_1) |f_1(y_1)| \, \d \mu(y_1) \quad \textrm{and}  \quad \B(f_2)(x_2) = \int_{Y}   b(x_2, y_2) |f_2(y_2)| \, \d \nu(y_2),                                  $$
where $a$ and $b$ are non-negative measurable functions. Suppose that $\| \A \|_{L^1 \to L^{1, \infty}} \leq C$, i.e. it holds that
$$ \mu(\{  x_1 \in X ; |\A(f_1)(x_1)|  > \lambda \}) \leq C \, \frac{\| f_1 \|_1}{\lambda}, \quad \forall \lambda > 0, \ f_1 \in L^1(X, \mu). $$
Moreover, assume that there exists a constant $M > 0$ such that
$$ \sup_{y_2 \in Y} \int_{Y} b(x_2, y_2) \, \d \nu(x_2) \leq M. $$
Then, for the following tensor product type operator $\T$, defined on $L^1(X \times Y, \mu \otimes \nu)$,
$$ \T(f)(x_1, x_2) =  \iint_{X \times Y} a(x_1, y_1) \, b(x_2, y_2) \, |f(y_1, y_2)|  \, \d \mu(y_1) \d \nu(y_2),$$
we have $\| \T \|_{L^1 \to L^{1, \infty}} \leq C \, M$.
\end{lem}

\begin{proof} Given $f \in L^1(X \times Y, \mu \otimes \nu)$, notice that $\int_{Y} b(x_2, y_2) \, |f(\cdot, y_2)| \, \d \nu(y_2) \in L^1(X)$ for a.e.-$\nu$ $x_2 \in Y$. Therefore the weak type $(1, 1)$ estimate of $\A$ implies that
$$
\mu(\{  x_1 \in X ; |\T(f)(x_1, x_2)|  > \lambda \}) \leq \frac{C}{\lambda} \int_{X} \! \bigg(\! \int_{Y} b(x_2, y_2) \, |f(y_1, y_2)| \, \d \nu(y_2) \bigg) \d \mu(y_1)
$$
holds for a.e.-$\nu$ $x_2 \in Y$.
Combining this with Fubini's theorem, we yield that
\begin{align*}
&(\mu \otimes \nu) (\{  (x_1 ,x_2) \in X \times Y ; \, |\T(f)(x_1, x_2)|  > \lambda \})\\
&= \int_{Y} \mu(\{  x_1 \in X ; \, |\T(f)(x_1, x_2)|  > \lambda \}) \, \d \nu(x_2)\\
&\leq \frac{C}{\lambda} \int_{Y} \bigg( \int_{X} \bigg( \int_{Y} b(x_2, y_2) \, |f(y_1, y_2)| \, \d \nu(y_2) \bigg) \, \d \mu(y_1) \bigg) \,  \d \nu(x_2).
\end{align*}
Then changing the order of integration, by the assumption of $b$, this lemma is proved.
\end{proof}

\medskip

\renewcommand{\theequation}{\thesection.\arabic{equation}}
\section{Standard method of localization for the Riesz transforms} \label{S2}
\setcounter{equation}{0}

\medskip

In this section, we will study the local part of the Riesz transform $\RR = \nabla_\M (-\Delta_\M)^{-\frac 12}$. To begin with, we represent $\RR$ as the sum of the local part $\RR_{\loc}$ and the part at infinity $\RR_\infty$. This step proceeds along the standard line, which has been widely adopted by other authors. In previous works mentioned in the end of Section \ref{SS12}, %\cite{LS17, LSW16},
the Calder\'on-Zygmund theory has been successfully applied to analyze $\RR_{\loc}$, mainly due to the availability of upper bounds for $|\nabla p_t^\M|$ and $|\nabla \nabla p_t^\M|$. However, in our setting, we do not have a (complete) pointwise upper bound of $|\nabla_\M p_t^\M|$, which motives the adoption of the method developed in \cite{CD99}.

\medskip

\subsection{Weak $(1, 1)$ of $\RR_{\loc}$ on general weighted Riemannian manifolds}

The localization procedure could be done on a general non-compact complete, weighted Riemannian manifold $\NN$ in the sense of \cite{G09} satisfying $(ND_{\loc})$ w.r.t. a quasi-distance $\r$. By the definition of $\r$, there exists a constant $\omega \geq 1$ such that:
\begin{align} \label{QD}
\r(\xx, \zz) \leq \omega \, ( \r(\xx, \yy) + \r(\yy, \zz)), \qquad \forall \, \xx, \yy, \zz \in \NN.
\end{align}

Selecting a $\frac{1}{4\omega^2}$-maximal separated subset of $\NN$, we get a countable family of balls $\{ B^j = B(\xx^j, \frac{1}{2\omega}); \, \xx^j \in \NN, j \in \N \}$. It follows from ($ND_{\loc}$) that this family satisfies the following properties:
\begin{enumerate}[(a)]
\item $\NN = \bigcup_{j \in \N} B^j$;
\item the balls $B (\xx^j, \frac{1}{4 \omega^2})$ are pairwise disjoint;
\item there exists $N \in \N$, such that every point of $\NN$ is contained in at most $N$ balls $(2\omega^2 + \omega) B^j := B(\xx^j, \omega + \frac12)$.
\end{enumerate}

Fix a $C^\infty$ partition of unity $( \varphi_j )_{j \in \N}$ such that supp $\varphi_j \subseteq B^j$. Given any $1 \leq p < + \infty$, the definition of partition of unity and the property (c) imply that $\frac{1}{N^p} \leq \sum_{j} \varphi_j^p \leq 1$. Therefore for all $f \in C_0^\infty(\NN)$,
\begin{align} \label{pro of p}
 \| f \|_p^p \sim \sum_j \| f\varphi_j \|_p^p.
\end{align}

For $f \in C_0^\infty(\NN)$, we write
\begin{align} \label{estofrf}
\RR(f) = \sum_{j} \chi_{(2\omega^2 + \omega)B^j} \, \RR(f \varphi_j) + \sum_{j} \big( 1 - \chi_{(2\omega^2 + \omega)B^j} \big) \, \RR(f \varphi_j).
\end{align}

Let $R(\xx, \yy) = \frac{1}{\sqrt{\pi}} \int_{0}^{\infty} \nabla_\NN p_t^\NN(\xx, \yy) \, \frac{\d t}{\sqrt{t}}$, which is the kernel of the Riesz transform. Then from the construction, we have for $\xx \in \NN$,
\begin{align} \label{estofrin}
\left| \sum_{j} \big( 1 - \chi_{(2\omega^2 + \omega)B^j} \big) \, \RR (f \varphi_j)(\xx) \right| &\leq \sum_j \int_{\{ \yy; \, \r > 1\}} |R(\xx, \yy)| \, |f(\yy)| \, \varphi_j (\yy) \,\d \mu_\NN(\yy)\nonumber\\
& \leq \int_{\{ \yy; \, \r > 1 \}} |R(\xx, \yy)| \, |f(\yy)| \,\d \mu_\NN(\yy).
\end{align}
We define the local part of the Riesz transform
\begin{align}
\RR_{\loc}(f)(\xx) := \sum_{j} \chi_{(2\omega^2 + \omega) B^j} \, \RR(f \varphi_j)(\xx),
\end{align}
and the part at infinity
\begin{align}
\RR_{\infty}(f)(\xx) := \int_{\{ \yy; \, \r > 1 \}} |R(\xx, \yy)| \, |f(\yy)| \,\d\mu_\NN(\yy).
\end{align}
Hence, in order to prove \eqref{W11} on $\NN$, it is sufficient to prove that $\RR_{\loc}$ and $\RR_{\infty}$ are both of weak type $(1, 1)$. In this section we focus on $\RR_{\loc}$.

For the further developments, we address here a rather general theorem on the weak $(1, 1)$ of $\RR_{\loc}$.
Indeed, it is an analogy of  \cite[Theorem 1.2]{CD99} (see also \cite[Proposition 3]{LZ17}).

\begin{theo} \label{Rloc}
Let $\NN$ be a connected complete non-compact weighted Riemannian manifold satisfying $(ND_{\loc})$ w.r.t. the quasi-distance $\r$. Moreover, suppose that for any given $R_0 > 0$, the ball $B(\xx, R_0)$ equipped with the induced quasi-distance $\rho$ and measure $\d \mu_{\NN}$ satisfies $(ND)$, with a doubling constant independent of its center. Assume that the following small-time heat kernel upper bound is valid: for some $\theta \in (0, 1]$, there exists a constant $c \in (0, 1)$ such that
$$p_t^\NN(\xx, \yy) \leq c^{-1} \frac{1}{V_\NN \left(\xx, \sqrt{t} \right)} \exp{\bigg(- c \, \bigg(\frac{ \r(\xx, \yy)^2}{t}\bigg)^{\theta}\bigg)},  \quad \forall \, 0< t \leq 1, \, \xx, \yy \in \NN. \eqno(GUE_{\theta}') $$
Additionally, suppose that there exist some constant $a \ge \frac12(2\omega^2 + \omega + 1)$, where $\omega \ge 1$ is the constant in \eqref{QD}, and a bounded function $F : [0, +\infty) \to (0, +\infty)$, which satisfies
$$ \int_{0}^{1} F \bigg( \frac{1}{u} \bigg) \, \frac{\d u}{u} < +\infty,$$
such that for any $0< t \leq 1$, $s > 0$ and $\yy \in \NN$,
$$\int_{\{ \xx; \, \sqrt{t} \leq \r \leq a \}}  |\nabla_\NN p_s^\NN (\xx, \yy)| \, \d\mu_\NN (\xx) \leq \frac{1}{\sqrt{s}} F \bigg( \frac{t}{s} \bigg),
\eqno(WGE_{1}')$$
Furthermore we assume that $\RR_{\infty}$ is of weak type $(1, 1)$. Then $\RR_{\loc}$ is of weak type $(1, 1)$.
\end{theo}

\begin{remark} \label{remark2}
Before turning to the proof, let us make a few comments on the requirement $a \ge \frac12(2\omega^2 + \omega + 1)$. First, it is purely technical and could be relaxed to $a > 0$. To see this, it suffices to suitably modify the numerical constants initially chosen in the localization procedure $($involving the radius of balls used in the construction and the definition of $\RR_\infty$$)$ to adapt the value of $a$. While the remaining argument of the proof is unchanged. Second, in concrete settings, this requirement is usually inessential. For instance, on the product manifolds to be considered later, we can even achieve a stronger version of $(WGE_{1}')$ as follows. For any fixed $a > 0$,
there exists a bounded function $F_a : [0, +\infty) \to (0, +\infty)$, depending on $a$, which satisfies
$$ \int_{0}^{1} F_a \bigg( \frac{1}{u} \bigg) \, \frac{\d u}{u} < +\infty,$$
such that for any $0< t \leq 1$, $s > 0$ and $\yy \in \NN$,
$$
\int_{\{ \xx; \, \sqrt{t} \leq \r \leq a \}}  |\nabla_\NN p_s^\NN (\xx, \yy)| \, \d\mu_\NN (\xx) \leq \frac{1}{\sqrt{s}} F_a \bigg( \frac{t}{s} \bigg).
$$

\end{remark}

\begin{proof}
This proof is actually a straightforward adaption of the argument used in \cite[\S 4]{CD99}. We include it for the sake of clarity.
Assume for simplicity of exposition that $\omega = 1$ (the general case $\omega \ge 1$ can be handled accordingly but some notational inconvenience).

In view of the property (c), for any fixed $\xx \in \NN$, there are at most $N$ non-zero terms in the sum $\sum_{j} \chi_{3 B^j} \, \RR(f \varphi_j)(\xx)$. Consequently, for $\lambda > 0$,
\begin{align} \label{estofRl}
\mu_\NN \left( \left\{ \xx \in \NN; \, |\RR_{\loc} (f)(\xx)| > \lambda \right\} \right) &= \mu_\NN ( \{ \xx \in \NN; \, |\sum_{j} \chi_{3B^j} \RR(f \varphi_j)(\xx)| > \lambda \} )\nonumber\\
&\leq \sum_{j} \mu_\NN \left( \left\{ \xx \in 3B^j; \, |\RR (f\varphi_j)(\xx)| > \lambda/N \right\} \right) .
\end{align}

According to \eqref{pro of p} for $p = 1$, it suffices to show that there exists a universal constant $c > 0$, such that for all $\lambda >0$ and $j \in \N$
\begin{align} \label{estofRl2}
\mu_\NN ( \{ \xx \in 3 B^j; \, |\RR (f)(\xx)| > \lambda \} ) \leq c \,  \frac{\| f \|_1}{\lambda}, \quad \forall \, f \in C_0^\infty(B^j).
\end{align}

By assumption, each $B^j$ satisfies ($ND$) with a constant independent of $j$. From now on we drop the subscripts of $B^j$. We can apply the classical Calder\'on-Zygmund decomposition. There exists a constant $\kappa > 0$ (only depends on the constant in $(ND)$, therefore independent of $j$) such that, given $f \in C_0^\infty(B)$ and $\lambda > \frac{\| f \|_1}{\mu_\M(B)}$, we can decompose $f$ as
$$ f = g + b = g + \sum_i b_i, $$
and find a sequence of balls $B_i = B(\xx_i, r_i) \subseteq B$, satisfying
\begin{enumerate}[(i)]
\item $\|g \|_2^2 \leq \kappa \, \lambda \, \|f\|_1$;
\item supp $b_i \subseteq B_i$ and $\| b_i \|_1 \leq \kappa \, \lambda \, \mu_\NN (B_i)$;
\item $\sum_i \mu_\NN (2B_i) \leq \kappa \, \lambda^{-1} \| f \|_1$.
\end{enumerate}
It is clear that $r_i \le \diam (B) \le 1$.

Notice that
\begin{align} \label{Eb}
& \mu_\NN ( \{ \xx \in 3 B; \, |\RR (f)(\xx)| > \lambda \} )  \nonumber \\
& \leq \mu_\NN ( \{ \xx \in 3 B; \, |\RR (g)(\xx)| > \lambda/2 \} ) + \mu_\NN ( \{ \xx \in 3 B; \, |\RR (b)(\xx)| > \lambda/2 \} ).
\end{align}
Combining the $L^2$ boundedness of $\RR$ with (i), the first term is easy to estimate. Thus we only need to treat the second one. We write
\begin{align} \label{Eb1}
\RR(b) = \sum_{i} \RR( \chi_{5B} \, \e^{t_i\Delta_\NN}b_i) +  \RR \big( (1 - \chi_{5B}) \, \sum_{i} \e^{t_i\Delta_\NN} b_i \big) + \sum_i \RR((I - \e^{t_i \Delta_\NN}) \, b_i) ,
\end{align}
where $t_i = r_i^2$. Then
\begin{align} \label{Eb2}
& \mu_\NN ( \{ \xx \in 3B; \, |\RR (b)(\xx)| > \lambda/2 \} ) \leq \mu_\NN ( \{ \xx \in 3 B; \, |\RR \big( \sum_{i}  \chi_{5B} \, \e^{t_i\Delta_\NN} b_i \big)(\xx)| > \lambda/6 \} ) \nonumber \\
& \quad + \, \mu_\NN ( \{ \xx \in 3 B; \, | \RR \big( (1 - \chi_{5B}) \, \sum_{i} \e^{t_i\Delta_\NN} b_i \big)(\xx)| > \lambda/6 \} ) \nonumber \\
& \qquad + \, \mu_\NN ( \{ \xx \in 3 B; \, |\sum_i \RR((I - \e^{t_i \Delta_\NN}) \, b_i)(\xx)| > \lambda/6 \} ).
\end{align}

It follows from our assumption and $(ND)$ that the centered Hardy-Littlewood maximal operator is bounded on $L^2(5B)$. On the other hand, by construction, each $t_i \leq 1$. Therefore we can apply $(GUE_{\theta}')$. By the duality method used in  \cite[\S 3]{CD99}, we obtain
\begin{align*}
\big\| \sum_i \chi_{5B} \, \e^{t_i \Delta_\NN} b_i \big\|_2^2 \lesssim \lambda \, \| f \|_1.
\end{align*}
Hence the $L^2$ boundedness of $\RR$ and Chebyshev's inequality imply that
\begin{align} \label{Eb3}
\mu_\NN ( \{ \xx \in 3 B; \,  |\sum_{i} \RR( \chi_{5B} \e^{t_i \Delta_\NN} b_i)(\xx)| > \lambda/6 \} ) \lesssim \frac{\| f \|_1}{\lambda}.
\end{align}

As for $\RR \big( (1-\chi_{5B})  \, \sum_{i} \e^{t_i\Delta_\NN} b_i \big)$, we shall use $\RR_\infty$ to dominate it. In fact, for $\xx \in 3B$,
\begin{align*}
\left| \RR \big(( 1 - \chi_{5 B}) \,  \sum_{i}  \e^{t_i\Delta_\NN} b_i \big)(\xx)  \right| \leq \RR_{\infty} \big(  (1 - \chi_{5 B}) \, \sum_{i} \e^{t_i\Delta_\NN} b_i \big)(\xx)  \leq \RR_{\infty} \big( \sum_{i} |\e^{t_i\Delta_\NN} b_i| \big)(\xx).
\end{align*}
Combining this with the weak type $(1, 1)$ estimate of $\RR_\infty$, using the contractivity of  the heat semigroup $(\e^{t\Delta_\NN})_{t  > 0}$,
we obtain
\begin{align} \label{Eb4}
\mu_\NN ( \{ \xx \in 3 B; \, | \RR( (1 - \chi_{5 B}) \, \sum_{i} \e^{t_i\Delta_\NN} b_i)(\xx)| > \lambda/6 \} )
\lesssim \frac{\sum_i \| b_i \|_1}{\lambda} \lesssim \frac{ {\| f \|}_1}{\lambda},
\end{align}
where we have used the properties (ii) and (iii) in the last inequality.

Therefore, by \eqref{Eb}-\eqref{Eb4},  to get \eqref{estofRl2}, it remains to establish the following estimate:
\begin{align*}
\mu_\NN ( \{ \xx \in 3 B; \, |\sum_i \RR((I - \e^{t_i \Delta_\NN}) \, b_i)(\xx)| > \lambda/6 \} )  \lesssim \frac{\| f \|_1}{\lambda},
\end{align*}
which can be proved by the standard techniques. More precisely,  the left-hand side of the above inequality is majorized by
\begin{align*}
\mu_\NN ( \cup_i 2 B_i) + \mu_\NN ( \{ \xx \in 3 B \setminus (\cup_i 2 B_i); \, |\sum_i \RR((I - \e^{t_i \Delta_\NN}) \, b_i)(\xx)| > \lambda/6 \} ).
\end{align*}
Using the properties (ii) and (iii) again, via Chebyshev's inequality, it is enough to prove that

\begin{align} \label{estofRl3}
\int_{3 B \backslash 2 B_i } |\RR ((I - \e^{t_i \Delta_\NN}) \, b_i)(\xx)| \,\d \mu_\NN (\xx) \lesssim \| b_i  \|_1, \qquad \forall \, i.
\end{align}

Denote by $k_t$ the kernel of $\RR ( I - \e^{t \Delta_\NN})$. Since supp $b_i \subseteq B_i \subseteq B$, we have
\begin{align*}
& \int_{3B \backslash 2 B_i } |\RR ((I-\e^{t_i \Delta_\NN}) \, b_i)(\xx)| \,\d \mu_\NN(\xx) \\
 \leq & \int_{3B \backslash 2 B_i } \left( \int_{B_i} |k_{t_i} (\xx, \yy)|  \, |b_i(\yy)| \, \d \mu_\NN (\yy) \right) \,\d \mu_\NN (\xx)\\
 \leq & \int_{B_i} \left( \int_{\{ \xx; \, \sqrt{t_i} \leq \r \leq 2 \}} |k_{t_i} (\xx, \yy)| \,\d\mu_\NN (\xx) \right) |b_i(\yy)| \,\d\mu_\NN(\yy),
\end{align*}
where
\begin{align*}
k_t(\xx, \yy) = \frac{1}{\sqrt{\pi}}\int_{0}^{\infty} \nabla_\NN p_s^\NN(\xx , \yy) \left( \frac{1}{\sqrt{s}} - \frac{{\chi}_{\{s>t\}}}{\sqrt{s-t}}   \right) \, \d s.
\end{align*}

Combining ($WGE_{1}'$) with the equality above, we yield that
\begin{align*}
\int_{\{ \xx ; \, \sqrt{t} \leq \r \leq 2 \}} |k_t (\xx, \yy)| \,\d\mu_\NN (\xx) & \lesssim \int_{0}^{\infty} \left| \frac{1}{\sqrt{s}} - \frac{\chi_{\{s > t\}}}{\sqrt{s-t}}   \right| F \left(\frac{t}{s}\right) \,\frac{\d s}{\sqrt{s}}\\
 & \lesssim \int_{0}^{t} F \left(\frac{t}{s}\right) \, \frac{\d s}{s} + \int_{t}^{\infty} \left| \frac{1}{\sqrt{s}} - \frac{\chi_{\{s>t\}}}{\sqrt{s-t}}   \right| F \left(\frac{t}{s}\right) \,\frac{\d s}{\sqrt{s}}.
\end{align*}
By our assumptions on $F$, the above integral is bounded from above uniformly in $t > 0$. Thus \eqref{estofRl3} is proved. This ends the proof of Theorem \ref{Rloc}.
\end{proof}

\medskip

\subsection{Consequences of Theorem \ref{Rloc}} \label{SS32}

As an application of Theorem \ref{Rloc}, we establish the weak type $(1, 1)$ boundedness of $\RR_{\loc}$ defined on $\M = \H^n \times M$ in this subsection. Actually, we can extend this result to general direct product manifolds with exponential volume growth. Let $\E$ be a connected complete non-compact weighted Riemannian manifold. Slightly abusing notation, we denote by $r$ the geodesic distance on $\E$.

\begin{cor} \label{c1}
Let $M$ be a connected complete non-compact weighted Riemannian manifolds satisfying $(ND)$, $(GUE_{\theta})$ and $(WGE_{q, \theta})$ $($with $1 \leq q \leq 2$, $0 < \theta \leq 1$$)$ w.r.t. the quasi-distance $\widetilde{d}$. Moreover, suppose that for any given $R_0 > 0$, the ball $\widetilde{B}(x, R_0)$ equipped with the induced quasi-distance $\widetilde{d}$ and measure $\d \mu_M$ satisfies $(ND)$, with a doubling constant independent of its center. Assume that $\E$ satisfies $(ND_{\loc})$ and $(VE)$, whose heat kernel $p_t^\E$  has the small-time upper bound
$$  p_t^\E(\u , \u)  \leq \frac{C}{V_\E(\u, \sqrt{t})}, \quad \forall \, \u \in \E, \, \, 0< t \leq 1. $$
Furthermore, we assume that $\RR_\infty$ defined on $\E \times M$ is of weak type $(1, 1)$. Then $\RR_{\loc}$ defined on $\E \times M$ is of weak type $(1, 1)$.
\end{cor}

\begin{proof}
We denote $\E \times M$ briefly by $\W$. A quasi-distance $d$ defined on $\W$ is also given by \eqref{distance}. According to Theorem \ref{Rloc}, it is enough to verify the conditions ($GUE_{\theta}'$) and ($WGE_{1}'$) on $\W$ because the other assumptions hold obviously, taking into account Remark \ref{r0}.

We start with ($GUE_{\theta}'$). It follows from \cite[Theorem 1.1]{G97} (cf. also \cite[Chapter 16]{G09}) that the small-time upper bound of $p_t^{\E}$ and $(ND_{\loc})$ of $\E$ imply the following off-diagonal estimate: given
$\gamma > 4$,
\begin{align*}
p_t^{\E}(\u, \v)  \lesssim_\gamma \frac{1}{\sqrt{V_{\E}(\u, \sqrt{t})V_{\E}(\v, \sqrt{t})}} \exp \left( -\frac{r^2}{\gamma t}  \right), \qquad \forall \, t \in (0, 1], \ \u, \v \in \E.                                                                              \end{align*}

And it is clear that by $(VE)$,
$$  V_{\E}(\u, \sqrt{t}) \leq  V_{\E}(\v, \sqrt{t}+r) \leq C \, \e^{c \, \frac{r}{\sqrt{t}}} V_{\E}(\v, \sqrt{t}), \quad \forall \, t \in (0, 1], \, \, \u, \v \in \E.   $$

Thus we obtain the following classical small time upper bound of $p_t^{\E}$:
given %%%any
$\gamma > 4$, we have
\begin{align} \label{off diagonal}
p_t^{\E}(\u, \v)  \lesssim_\gamma \frac{1}{V_{\E}(\u, \sqrt{t})}  \exp\left( -\frac{r^2}{\gamma t}  \right), \qquad \forall \, t \in (0, 1], \ \u, \v \in \E.
\end{align}
Moreover, by the trivial inequality $s^{\theta} \le 1 + s$ for all $s \ge 0$ and $0 < \theta \le 1$, we get
\begin{align} \label{nHn}
p_t^{\E}(\u, \v)  \lesssim_\gamma \frac{1}{V_{\E}(\u, \sqrt{t})}  \exp\left( -\frac{1}{\gamma} \, \bigg( \frac{r^2}{t} \bigg)^{\theta} \right), \qquad \forall \, t \in (0, 1], \ \u, \v \in \E.
\end{align}

Now, for $t \in (0, 1]$ and $\X = (\u, \x), \Y = (\v, \y) \in \E \times M = \W$, the basic property of the heat kernel on direct product manifolds says that $p_t^{\W}(\X, \Y) = p_t^{\E}(\u, \v) \, p_t^{M} (\x, \y)$. Combining this with \eqref{nHn}, the assumption $(GUE_{\theta})$ on $M$ implies  that there exists a constant $c \in (0, 1)$ such that
\begin{align} \label{on W}
p_t^{\W}(\X, \Y) &\lesssim \frac{1}{V_{\E}(\u, \sqrt{t}) \widetilde{V}_{M}(\x, \sqrt{t})}  \exp{\bigg\{- c \, \bigg[ \bigg(\frac{ r(\u, \v)^2}{t}\bigg)^{\theta} + \bigg(\frac{ \widetilde{d}(\x, \y)^2}{t}\bigg)^{\theta} \bigg]\bigg\}}\nonumber\\
&\lesssim \frac{1}{V_{\W}(\X, \sqrt{t})}  \exp{\bigg(- c \, \bigg(\frac{ d(\X, \Y)^2}{t}\bigg)^{\theta}\bigg)},
\end{align}
where we have used in the last line the following trivial observation:
\begin{align*}
\{\Y; \, d(\X, \Y) < s \} = \{\v; \, r(\u, \v) < s \} \times \{\y; \, \widetilde{d}(\x, \y) < s\}, \qquad \forall \, s > 0.
\end{align*}
In other terms, ($GUE_{\theta}'$) is valid on $\W$.

We turn to ($WGE_{1}'$). Indeed, the requirement $a \ge \frac12(2\omega^2 + \omega + 1)$ in Theorem \ref{Rloc} is only a technical assumption. The following proof also works for any given $a > 0$ in our setting (see Remark \ref{remark2}). Since ($WGE_{1}'$) holds trivially if $t > a^2$, in the following we assume $t \le a^2$. Recalling the definition of $d$, we obtain for all $t \in [0, 1]$ and $\Y = (\v, \y) \in \W$,
\begin{align} \label{e of int}
& \int_{\{ \X; \, \sqrt{t} \leq d \leq a \}}  |\nabla_{\W} p_s^\W (\X, \Y)| \, \d \mu_\W (\X) \nonumber\\
&\leq \int_{ \left\{ \X; \,  \sqrt{t} \leq r \leq a , \, \, \widetilde{d} \leq a  \right\}}  |\nabla_{\W} p_s^\W (\X, \Y)| \, \d \mu_\W(\X) +  \int_{\left\{ \X; \, r \leq a ,\, \, \, \sqrt{t} \leq \widetilde{d} \leq a  \right\}}  |\nabla_{\W} p_s^\W (\X, \Y)| \, \d \mu_\W(\X) \nonumber\\
& \leq \int_{\{\u; \, \sqrt{t} \leq r \leq a\}} p_s^{\E} \d\mu_{\E} \! \int_{\{\x ; \, \widetilde{d} \leq a \}} |\nabla_M p_s^M | \d\mu_M + \int_{\{\u; \, \sqrt{t} \leq r \leq a\}} |\nabla_{\E} p_s^{\E}| \d\mu_{\E} \! \int_{\{\x ; \, \widetilde{d} \leq a \}} p_s^M \d \mu_M \nonumber\\
&  + \int_{\{\u; \, r \leq a\}} p_s^{\E} \d \mu_{\E}  \int_{\{\x ; \, \sqrt{t} \leq \widetilde{d} \leq a \}} |\nabla_M p_s^M| \d\mu_M + \int_{\{\u; \,  r \leq a\}} |\nabla_{\E} p_s^{\E}| \d\mu_{\E}  \int_{\{\x ; \, \sqrt{t} \leq \widetilde{d} \leq a \}} p_s^M \d\mu_M.
\end{align}

This inequality implies that, in order to verify ($WGE_{1}'$) on $\W$, we only need to estimate the integrals on $\E$ and $M$ separately. We first consider the terms involving $\E$. Our proof is mainly based on the weighted $L^2$-norm estimate, which has become a standard tool in heat kernel theory. We refer the reader to the monograph \cite{G09}
for more details. Given $D > 2$, we define the weighted integrals of the heat kernel and its gradient as follows: for $s > 0$ and $\v \in \E$,
$$ E(\v, s) :=  \int_{\E} \,  p_s^{\E}(\u, \v)^2 \exp \left( \frac{r^2}{Ds} \right)  \, \d \mu_{\E}(\u),$$
and
$$E_1(\v, s) := \int_{\E} \, |\nabla_{\E} p_s^{\E}(\u, \v)|^2 \exp\left( \frac{r^2}{Ds} \right)  \, \d \mu_{\E}(\u).                                                                                             $$

Recall that a basic fact about the heat kernel is that $p_s^{\E}(\v, \v)$ is decreasing w.r.t. $s > 0$ for any given $\v \in \E$.
By \cite[Theorem 16.3]{G09} (see also \cite[Theorem 2.1]{G97}), it follows from the small-time upper bound of $p_t^\E$ and $(ND_{\loc})$ that
\begin{align}  \label{nEn1}
E(\v, s) \lesssim \frac{1}{V_{\E}(\v, \sqrt{s})}, \qquad \forall \, s \in (0, 1], \ \v \in \E.
\end{align}

Moreover, let $\lambda_1 \geq 0$ denote the bottom of the spectrum of $(-\Delta_\E)$. Combining \eqref{nEn1} with the integral maximum principle (see e.g. \cite[Theorem 12.1]{G09}),
we obtain (see also \cite[(15.25)]{G09})

\begin{align} \label{weighted estimate}
E(\v, s) \lesssim \frac{\e^{-2\lambda_1 s}}{ V_{\E}(\v, \min \{\sqrt{s}, 1\}) } , \qquad \forall \, s > 0, \ \v \in \E.
\end{align}
Then  \cite[Theorem 1.1]{G95} (notice that the main results and proofs therein remain valid in the setting of weighted Riemannian manifolds without modifications) implies that
\begin{align} \label{nEn2}
E_1(\v, s) \lesssim \frac{\e^{-2\lambda_1s}}{s V_{\E}(\v, \min \{\sqrt{s}, 1\}) }, \qquad \forall \, s > 0, \ \v \in \E.
\end{align}

On the other hand, by the method of decomposition in annuli, $(VE)$ implies that
\begin{align} \label{nEn3}
\int_{ \{ \u; \, \sqrt{t} \leq r \leq a  \}} \e^{-\frac{r^2}{Ds}} \, \d \mu_{\E}(\u)  \lesssim_a \e^{-\frac{t}{2Ds}} V_{\E}(\v, \min \{\sqrt{s}, 1\}), \qquad \forall \, s > 0, \ t \in [0, 1].
\end{align}
Here and in \eqref{nEn4}-\eqref{truncated estimate} below, the implicit constants may increase exponentially w.r.t. $a$.
Combining
this with \eqref{weighted estimate} and \eqref{nEn2}, by Cauchy-Schwarz inequality, we obtain for all $s > 0$, $t \in [0, 1]$ and $\v \in \E$,
\begin{align} \label{nEn4}
\int_{ \{ \u ; \, \sqrt{t} \leq r \leq a   \}}  p_s^{\E}(\u, \v) \, \d\mu_{\E}(\u) \lesssim_a \exp \left(-\frac{t}{4Ds}- \lambda_1s\right),
\end{align}
and
\begin{align}  \label{truncated estimate}
\int_{ \{ \u ; \,  \sqrt{t} \leq r \leq a   \}} |\nabla_{\E} p_s^{\E}(\u, \v)| \, \d\mu_{\E}(\u) \lesssim_a  \frac{1}{\sqrt{s}} \exp\left(-\frac{t}{4Ds}-\lambda_1s\right).
\end{align}

Next, we turn to the analysis of the terms in \eqref{e of int} involving $M$, which is much simpler. A direct computation shows that there exists $c > 0$ such that for all $s > 0$, $t \in [0, 1]$ and $\y \in M$,
\begin{align} \label{nEn5}
\int_{ \{ \x ; \, \sqrt{t} \leq \widetilde{d} \leq a   \}}  p_s^{M}(\x, \y) \, \d\mu_M (\x) \lesssim \exp \left(- c \left( \frac{t}{s} \right)^\theta \right),
\end{align}
and
\begin{align} \label{nEn6}
\int_{ \{ \x ; \,  \sqrt{t} \leq \widetilde{d} \leq a   \}} |\nabla_{M} p_s^{M}(\x, \y)| \, \d\mu_M (\x) \lesssim  \frac{1}{\sqrt{s}} \exp \left(- c \left( \frac{t}{s} \right)^\theta \right).
\end{align}

In conclusion, collecting
\eqref{nEn4}-\eqref{nEn6}, from \eqref{e of int}, ($WGE_{1}'$) holds on $\W$ for any $a > 0$ and $F_a (h) = c_a^{-1} \e^{-c_a h^{\theta}}$, which completes our proof.
\end{proof}

\medskip

\begin{remark}
The weighted norm $E(\cdot, t)$ is an important tool to establish pointwise upper bounds of the heat kernel, based on the following inequality (see \cite{G95})
\begin{align*}
p_t^{\E}(\u, \v) \leq \sqrt{E(\u, t/2)E(\v, t/2)}\exp \left( - \frac{r^2}{2Dt} \right).
\end{align*}
Combining this with \eqref{weighted estimate} and $(ND_{\loc})$, the heat kernel $p_t^{\E}$ satisfies the following local estimate: given $r_0 > 0$, for all $t > 0$ and $\u, \v \in \E$ with $r = r(\u, \v) \leq r_0$,
\begin{align} \label{Local Estimate}
p_t^{\E}(\u, \v) \lesssim_{r_0} \frac{1}{V_{\E}(\u, \min \{\sqrt{t}, 1\})} \exp \left(- \frac{r^2}{2Dt}  - \lambda_1t   \right),
\end{align}
Then by the similar argument used in \eqref{on W}, there exists $c > 0$ such that for all $t > 0$ and $\X = (\u, \x), \Y = (\v, \y) \in \W$ with $d = d(\X, \Y) \le 1$,
\begin{align} \label{local on W}
p_t^{\W}(\X, \Y) \lesssim \frac{1}{V_{\W}(\X, \min \{\sqrt{t}, 1\})} \exp \left( -c \left( \frac{d^2}{t} \right)^\theta  \right).
\end{align}
In Section \ref{S4}, this estimate will be used in the study of the heat maximal operator.
\end{remark}

\medskip

It is clear that $\H^n$ is an example of such $\E$ (see \eqref{small time gaussian}). Combining Corollary \ref{c1} with Theorems \ref{t3} and \ref{t4} in the next section, we conclude that
\begin{cor} \label{c2}
$\RR_{\loc}$ defined on $\H^n \times M$ is of weak type $(1, 1)$.
\end{cor}
The weak $(1, 1)$ property of $\RR_{\infty}$ on a general $\W$ is a challenging problem. In Section \ref{S6}, we shall give a sufficient condition, whose put-forward is inspired by the case $\M = \H^n \times M$.

\medskip

\renewcommand{\theequation}{\thesection.\arabic{equation}}
\section{Proof of Theorem \ref{nt1}: the part at infinity}
\setcounter{equation}{0} \label{S3}

\medskip
In this section, we shall prove the weak type $(1, 1)$ estimate of the part at infinity $\RR_\infty$.  For this purpose, set in the sequel for $\X = (\u, \x), \Y = (\v, \y) \in \H^n \times M$,
\begin{gather*}
R_1(\X, \Y) := \frac{1}{\sqrt{\pi}} \int_0^{\infty}  \left| \nabla_{\H^n} p_t^{\H^n}(\u, \v) \right| \,  p_t^{M}(\x, \y) \, \frac{\d t}{\sqrt{t}}, \\ R_2(\X, \Y) := \frac{1}{\sqrt{\pi}} \int_0^{\infty} p_t^{\H^n}(\u, \v) \, \left| \nabla_M p_t^{M}(\x, \y) \right| \, \frac{\d t}{\sqrt{t}}.
\end{gather*}

Now we define $\RR_{i, \infty} \, (i= 1,2)$ as follows:
\begin{align}
\RR_{i, \infty}f(\X) := \int_{\{\Y ; \, d(\X, \Y) \geq 1 \}} R_i(\X, \Y) \, |f(\Y)| \, \d\mu_\M (\Y).
\end{align}
Notice that $\RR_{\infty}(f) \leq \RR_{1, \infty}(f) + \RR_{2, \infty}(f)$. Hence it
is sufficient to prove the operators $\RR_{i, \infty} \, (i= 1,2)$ are of weak type $(1, 1)$.
\medskip

\subsection{Weak $(1, 1)$ of $\RR_{1, \infty}$}
We start with $\RR_{1, \infty}$, which is easier to handle for the reason that $R_1(\X, \Y)$ has an explicit pointwise upper bound.

\begin{theo} \label{t3}
 $\RR_{1, \infty}$ defined  on $\M$ is of weak type $(1, 1)$.
\end{theo}

\begin{proof} Since $\{ (\X, \Y) \in \M^2 ; d > 1 \} \subseteq S_1 \cup S_2$ with
\begin{align}
S_1 := \{ (\X, \Y) \in \M^2; \, \widetilde{d} > 1, \, r \leq 1 \}, \qquad S_2 := \{ (\X, \Y) \in \M^2; \, r > 1 \},
\end{align}
it suffices to show that the integral operators with kernels $R_1(\X, \Y) \, \chi_{S_j}$ ($j = 1, 2$) are of weak type $(1, 1)$.

We start with the kernel $R_1(\X, \Y) \, \chi_{S_1}$. And we shall prove that
\begin{align} \label{nBn}
\sup_{\Y \in \M} \int_{\M} R_1(\X, \Y) \, \chi_{S_1}(\X, \Y) \, \d\mu_\M(\X) < +\infty,
\end{align}
then the integral operator with the kernel $R_1(\X, \Y) \, \chi_{S_1}$ is bounded on $L^1(\M)$. Indeed, we can write
\begin{align} \label{nBn1}
0 < R_1(\X, \Y) &\lesssim  \int_{0}^{\infty} \left| \nabla_{\H^n} p_t^{\H^n}(\u, \v) \right| \, p_t^{M}(\x, \y) \, \frac{\d t}{\sqrt{t}}  \nonumber \\
& \leq  \int_{0}^{\infty} \left| \nabla_{\H^n} p_t^{\H^n}(\u, \v) \right| \, \d t \cdot \sup_{t>0} \frac{1}{\sqrt{t}} p_t^{M}(\x, \y).
\end{align}

As a consequence of \eqref{gr of hk}, for $r \in (0 , 1]$, we have
\begin{align} \label{nBn2}
\int_{0}^{\infty} \left| \nabla_{\H^n} p_t^{\H^n}(\u, \v) \right| \, \d t & \lesssim r \int_{0}^{\infty} t^{-\frac{n+2}{2}} \exp{ \left( -\frac{r^2}{C t} \right) } \, \d t %%%\nonumber \\
%%%& \
\lesssim r^{-(n - 1)} \lesssim \frac{r}{V_{\H^n}(\v, r)},
\end{align}
where we have used \eqref{ve} in the last line.

On the other hand,  using the symmetric property of the heat kernel, ($GUE_{\theta}$) and \eqref{nd} imply that
\begin{align} \label{nBn3}
\sup_{t>0}  \frac{1}{\sqrt{t}} \, p_t^M (\x, \y) & = \sup_{t>0}  \frac{1}{\sqrt{t}} \, p_t^M (\y, \x) \nonumber\\
& \lesssim \sup_{t > 0}  \frac{1}{\sqrt{t}} \,
  \frac{1}{\widetilde{V}_M \left(\y, \sqrt{t} \right)} \, \e^{-c \, \left( \frac{\widetilde{d}^2}{t}\right)^\theta} \nonumber \\
& \lesssim \sup_{t > 0}  \frac{\widetilde{d}}{\sqrt{t}} \, \bigg(1 + \frac{\widetilde{d}}{\sqrt{t}}\bigg)^\upsilon \,  \e^{-c \, \left( \frac{\widetilde{d}^2}{t}\right)^\theta}  \cdot \frac{1}{\widetilde{d}} \, \frac{1}{ \widetilde{V}_M \big(\y, \widetilde{d} \big) } \nonumber\\
&\lesssim \frac{1}{\widetilde{d}} \, \frac{1}{ \widetilde{V}_M \big(\y, \widetilde{d} \big) }.
\end{align}
Then by the standard method of decomposition in annuli, the volume doubling property implies that
\begin{align} \label{nBn4}
& \sup_{\y \in M} \int_{\{ \x ; \,  \widetilde{d} > 1 \} } \frac{1}{\widetilde{d}} \, \frac{1}{ \widetilde{V}_M \big( \y, \widetilde{d} \big) }  \, \d\mu_M (\x) \nonumber \\
&  = \sup_{\y \in M} \sum_{j = 0}^{+\infty} \int_{\{ \x ; \,  2^j < \widetilde{d} \le 2^{j+1} \} } \frac{1}{\widetilde{d}} \, \frac{1}{ \widetilde{V}_M \big(\y, \widetilde{d} \big) }  \, \d\mu_M(\x)< +\infty.
\end{align}
Similarly, the local volume doubling property on $\H^n$ (cf. ($ND_{\loc}$)) implies that
\begin{align} \label{nBn5}
\sup_{\v \in \H^n} \int_{\{ \u ; \,  r \le 1 \} } \frac{r}{V_{\H^n}(\v, r)} \, \d\mu_{\H^n}(\u) < + \infty.
\end{align}

Hence, \eqref{nBn} is a direct consequence of \eqref{nBn1}-\eqref{nBn5}.

We now turn to study $R_1(\X, \Y) \, \chi_{S_2}$. For a fixed $\varepsilon \in (0, \frac{1}{64})$, we write
\begin{align*}
0 < R_1(\X, \Y) &\lesssim  \int_{0}^{\infty} \left| \nabla_{\H^n} p_t^{\H^n}(\u, \v) \right| \, p_t^{M}(\x, \y) \, \frac{\d t}{\sqrt{t}}\\
& \leq  \int_{0}^{\infty} \left| \nabla_{\H^n} p_t^{\H^n}(\u, \v) \right| \, \e^{\varepsilon \frac{r}{t}} \, \d t \cdot \sup_{t>0}  \frac{1}{\sqrt{t}} \, p_t^{M}(\x, \y) \, \e^{-\varepsilon \frac{r}{t}}.
\end{align*}

For $r > 1$, the gradient estimate \eqref{gr of hk} implies that
\begin{align*}
&\int_{0}^{\infty} \left| \nabla_{\H^n} p_t^{\H^n}(\u, \v) \right| \, \e^{\varepsilon \frac{r}{t}} \, \d t \\ &\lesssim \e^{-(n-1) \, r}   \sqrt{r} \int_{0}^{\infty} \left( \frac{r}{t} \right)^{\frac12} \, \left( 1 + \frac{r}{t} \right)^{\frac{n-1}{2}} \, \exp\left\{ -\frac{r}{4}\left((n-1)\sqrt{\frac{t}{r}}- \sqrt{\frac{r}{t}} \right)^2 + \varepsilon \frac{r}{t} \right\} \, \frac{\d t}{t}.
\end{align*}
Changing variables by $t = r \, \lambda$, the last integral is equal to
\begin{align*}
& \int_{0}^{\infty} \lambda^{-\frac12} \, ( 1 + \lambda^{-1})^{\frac{n-1}{2}} \,  \exp\left\{-\frac{r}{4} \frac{((n-1)\lambda - 1  )^2}{\lambda}+ \frac{\varepsilon}{\lambda} \right\} \, \frac{\d\lambda}{\lambda}\\
&= \int_{(n - 1) \lambda - 1 \ge \frac{1}{2}} + \int_{- 1 < (n - 1) \lambda - 1 \le - \frac{1}{2}} + \int_{-\frac{1}{2} < (n - 1) \lambda - 1 < \frac{1}{2}};
\end{align*}
and the standard Laplace's method says that the main contribution comes from the third term, which is bounded by a multiple of $r^{-\frac{1}{2}}$. We conclude that when $r > 1$,
\begin{align} \label{Im1A}
\int_{0}^{\infty} \left| \nabla_{\H^n} p_t^{\H^n}(\u, \v) \right| \, \e^{\varepsilon \frac{r}{t}} \, \d t  \lesssim_{\varepsilon}  \e^{-(n - 1) \, r}.
\end{align}
According to Lemma \ref{weak type on hn}, the integral operator on $\H^n$, with the kernel defined by the left-hand side of \eqref{Im1A},  is of weak type $(1, 1)$.

On the other hand, by a suitable modification to the proof of \eqref{nBn3}, we can show that for $r > 1$,
\begin{align} \label{nBn6}
\sup_{t>0}  \frac{1}{\sqrt{t}} p_t^{M}(\x, \y) \e^{-\varepsilon \frac{r}{t}}  \lesssim_{\varepsilon}  \min \bigg\{ \frac{1}{\widetilde{V}_M \big(\y, 1 \big)},  \frac{1}{\widetilde{d}} \, \frac{1}{\widetilde{V}_M \big(\y, \widetilde{d} \big) } \bigg\}.
\end{align}
We claim that the integral operator on $M$, with the kernel defined by the left-hand side of the last inequality, is bounded on $L^1$. Indeed, it can be easily verified that
\begin{align} \label{nBn7}
\sup_{\y \in M} \int_{M} \min \bigg\{ \frac{1}{\widetilde{V}_M \big(\y, 1 \big)},  \frac{1}{\widetilde{d}} \, \frac{1}{\widetilde{V}_M \big(\y, \widetilde{d} \big) } \bigg\} \, \d\mu_M(\x)  < +\infty.
\end{align}
Hence, by Lemma \ref{weak type 11}, the integral operator with the kernel $R_1(\X, \Y) \, \chi_{S_2}$ is of weak type $(1, 1)$. In conclusion, the weak type $(1, 1)$ estimate of $\RR_{1, \infty}$ is established.
\end{proof}

\medskip

\subsection{Weak $(1, 1)$ of $\RR_{2, \infty}$}
In this section, we turn to the analysis of $\RR_{2, \infty}$. Compared with $\RR_{1, \infty}$, the main difficulty here is the lack of a pointwise upper bound of $| \nabla_M p_t^{M} |$. By \eqref{GUE2}, without loss of generality, we only assume $(WGE_{1, \theta})$ on $M$, i.e. there exist constants $c, c' > 0$ and $0 < \theta \leq 1$ such that for all $t > 0$, $\y \in M$,
$$\int_M \left| \nabla_M p_t^{M}(\x, \y) \right| \, \exp{\bigg( c \,  \bigg(\frac{\widetilde{d}(\x, \y)^2}{t} \bigg)^{\theta}\bigg)} \, \d\mu_M(\x) \leq \frac{c'}{\sqrt{t} }. $$
In particular, as its corollary
\begin{align} \label{whole space gradient norm}
\int_M \left| \nabla_M p_t^{M}(\x, \y) \right| \, \d\mu_M (\x) \lesssim \frac{1}{\sqrt{t} }.
\end{align}

\begin{theo} \label{t4}
$\RR_{2, \infty}$ defined on $\M$ is of weak type $(1, 1)$.
\end{theo}

\begin{proof} As presented in the preceding subsection, it suffices to show that the integral operators with kernels $R_2(\X, \Y) \, \chi_{S_j}$ ($j = 1, 2$) are of weak type $(1, 1)$.

We start with the analysis of $R_2(\X, \Y) \, \chi_{S_1}$. Notice that
\begin{align} \label{nZn1}
0 < R_2(\X, \Y) &\lesssim \int_{0}^{\infty}  p_t^{\H^n}(\u, \v) \, | \nabla_M p_t^{M}(\x, \y) | \, \frac{\d t}{\sqrt{t}} \nonumber \\
& \leq  \sup_{t>0} \sqrt{t} \, p_t^{\H^n}(\u, \v) \cdot \left( \int_{0}^{\infty} {\left| \nabla_M p_t^{M}(\x, \y) \right|} \, \frac{\d t}{t}  \right).
\end{align}
The heat kernel estimate on $\H^n$ (see \eqref{local estimate}) implies that when $r \leq 1$,
\begin{align} \label{nZn2}
\sup_{t>0} \sqrt{t} \, p_t^{\H^n}(\u, \v) \lesssim \sup_{t>0} t^{-\frac{n}{2}+\frac{1}{2}} \e^{-\frac{r^2}{C t}} \lesssim r^{-(n-1)} \lesssim \frac{r}{V_{\H^n}(\v, r)}.
\end{align}

On the other hand, the assumption $(WGE_{1, \theta})$ on $M$ implies that there exists a constant $c > 0$, such that for all $\y \in M$ and $t > 0$,
$$\int_{\{\x ; \, \widetilde{d} > 1 \}} \left| \nabla_M p_t^{M}(\x, \y) \right| \, \d\mu_M(\x) \lesssim \frac{\e^{-c \, t^{ - \theta}}}{\sqrt{t} }. $$

Hence
\begin{align} \label{nZn3}
& \sup_{\y \in M} \int_{\{\x ; \, \widetilde{d} > 1 \}} \left( \int_{0}^{\infty} {\left| \nabla_M p_t^{M}(\x, \y) \right|} \, \frac{\d t}{t}  \right) \, \d\mu_M(\x) \lesssim \int_{0}^{\infty} \frac{\e^{-c \, t^{ - \theta}}}{t^{3/2} } \, \d t < + \infty.
\end{align}

It deduces from \eqref{nZn1}-\eqref{nZn3}  and \eqref{nBn5} that
\[
\sup_{\Y \in \M} \int_{\M} R_2(\X, \Y) \,  \chi_{S_1}(\X, \Y) \, \d \mu_\M(\X) < +\infty.
\]
Therefore the integral operator with the kernel $R_2(\X, \Y) \, \chi_{S_1}$ is bounded on $L^1$.

We turn to study $R_2(\X, \Y) \, \chi_{S_2}$. For a fixed $\varepsilon \in (0, \frac{1}{64})$, we write
\begin{align*}
0 < R_2(\X, \Y) & \lesssim  \int_{0}^{\infty}  p_t^{\H^n}(\u, \v) \, | \nabla_M p_t^{M}(\x, \y) | \, \frac{\d t}{\sqrt{t}} \\
& \leq  \sup_{t>0} \sqrt{t} \, p_t^{\H^n}(\u, \v) \, \e^{\varepsilon \, \frac{r}{t}} \cdot \left( \int_{0}^{\infty} {\left| \nabla_M p_t^{M}(\x, \y) \right|} \, \e^{-\varepsilon \frac{r}{t}} \, \frac{\d t}{t}  \right).
\end{align*}

The heat kernel estimate \eqref{hk of hn} implies that when $r > 1$,
\begin{align} \label{nZn4}
& \sup_{t>0} \sqrt{t} \, p_t^{\H^n}(\u, \v) \, \e^{\varepsilon \frac{r}{t}} \nonumber\\
&\lesssim \sup_{t > 0}  \frac{r}{t} \left( 1 + \frac{r}{t} \right)^{\frac{n-3}{2}} \exp \left\{-\frac{r}{4}\left((n-1)\sqrt{\frac{t}{r}}- \sqrt{\frac{r}{t}} \right)^2 +  \varepsilon \frac{r}{t} \right\} \cdot \e^{-(n - 1) \, r}\nonumber\\
& \lesssim \e^{-(n - 1) \, r} \, \sup_{\lambda > 0} \lambda^{-1}(1 + \lambda^{-1})^{\frac{n-3}{2}} \exp\left\{-\frac{1}{4} \frac{((n-1)\lambda-1)^2}{\lambda} + \frac{\varepsilon}{\lambda} \right\} \nonumber\\
&\lesssim_{\varepsilon} \e^{-(n - 1) \, r}.
\end{align}
In the last step we have used the following simple observation: the continuous function
\[
g_{\varepsilon}(\lambda):= \lambda^{-1}(1 + \lambda^{-1})^{\frac{n-3}{2}} \exp\left\{-\frac{1}{4} \frac{((n-1)\lambda-1)^2}{\lambda} + \frac{\varepsilon}{\lambda} \right\},
\]
which is defined on $(0, \, +\infty)$, satisfies $\lim_{\lambda \to 0^+} g_{\varepsilon}(\lambda) = \lim_{\lambda \to +\infty} g_{\varepsilon}(\lambda) = 0$.

Then we set
\begin{align*}
& T(\X, \Y)  :=  \e^{-(n - 1) \, r} \cdot \left( \int_{0}^{\infty} {\left| \nabla_M p_t^{M}(\x, \y) \right|} \, \e^{-\varepsilon \frac{r}{t}} \, \frac{\d t}{t}  \right) \chi_{\{ r > 1 \}},\\
& \T f(\X)  := \int_{\M} T(\X, \Y) \, |f(\Y)| \, \d\mu_\M (\Y).
\end{align*}
It is enough to prove that the operator $\T$ is of weak type $(1, 1)$.

We define for $j \in \N$,
\begin{align*}
T_j(\X, \Y) := T(\X, \Y) \, \chi_{\{ 2^j  < \, r \leq 2^{j+1}  \}} ,\qquad
\T_j f(\X) := \int_{\M} T_j(\X, \Y)  \, |f(\Y)| \,\d \mu_\M(\Y).
\end{align*}
Then it is clear that
\begin{align*}
T_j(\X, \Y) \le \e^{-(n - 1) \, r} \cdot \left( \int_{0}^{\infty} {\left| \nabla_M p_t^{M}(\x, \y) \right|} \, \e^{-\varepsilon \frac{2^j}{t}} \, \frac{\d t}{t} \right).
\end{align*}

We claim that $\| \T_j \|_{L^1\to L^{1,\infty}} \lesssim_{\varepsilon} 2^{-j/2}$. Combining Lemma \ref{weak type on hn} and Lemma \ref{weak type 11}, it remains to show that
\begin{align*}
\sup_{\y \in M} \int_{M} \left( \int_{0}^{\infty} {\left| \nabla_M p_t^{M}(\x, \y) \right|} \,  \e^{-\varepsilon \frac{2^j}{t}} \, \frac{\d t}{t} \right) \, \d\mu_M(\x) \lesssim_{\varepsilon} 2^{-j/2}, \qquad \forall \, j \in \N.
\end{align*}
In fact, changing the order of integration and using \eqref{whole space gradient norm} lead to
\begin{align*}
&\int_{M} \left( \int_{0}^{\infty} {\left| \nabla_M p_t^{M}(\x, \y) \right|} \, \e^{-\varepsilon \frac{2^j}{t}} \, \frac{\d t}{t} \right) \, \d\mu_M(\x)\\
& = \int_{0}^{\infty} \left( \int_{M}  {\left| \nabla_M p_t^{M}(\x, \y)\right|} \, \d\mu_M(\x) \right) \,  \e^{-\varepsilon \frac{2^j}{t}} \, \frac{\d t}{t}\\
& \lesssim \int_{0}^{\infty} \frac{\e^{-\varepsilon \frac{2^j}{t}}}{t^{3/2}} \, \d t = 2^{-j/2} \, \int_{0}^{\infty} \frac{\e^{ -\varepsilon s^{-1}}}{s^{3/2}} \, \d s \lesssim_{\varepsilon} 2^{-j/2}.
\end{align*}
Then  \cite[Lemma 2.3]{SW69} implies the weak type $(1, 1)$ of $\T = \sum_{j = 0}^{\infty} \T_j$. Hence the integral operator with the kernel $R_2(\X, \Y) \, \chi_{S_2}$ is of weak type $(1, 1)$. The proof is completed.
\end{proof}

\medskip

\renewcommand{\theequation}{\thesection.\arabic{equation}}
\section{The heat maximal operator on $\M$} \label{S4}
\setcounter{equation}{0}

\medskip
In this section, we shall establish the weak type $(1, 1)$ property of the heat maximal operator on $\M$,   %%%. Recall the heat maximal operator
 which is defined by
\begin{align*}
\HH (f)(\X) = \sup_{t>0} |\e^{t \Delta_\M}f(\X)|, \quad \forall \, f \in L^p (\M), \, 1 \leq p \leq + \infty.
\end{align*}
The basic strategy is the same as that used in the study of Riesz transform. We first split $\HH$ as the sum of the local part and the part at infinity. Then we prove they are of weak type $(1, 1)$ separately.

Let us first write $\HH$ as
\begin{align*}
\HH (f)(\X) &\leq \sup_{t>0} \left| \int_{\{\Y; \, d \leq 1 \}} p_t^\M (\X, \Y) f(\Y) \,\d \mu_\M(\Y) \right| + \sup_{t>0} \left| \int_{\{\Y; \, d > 1 \}} p_t^\M(\X, \Y) f(\Y) \,\d \mu_\M(\Y) \right| \\
&\leq \sup_{t>0} \int_{\{\Y; \, d \leq 1 \}} p_t^\M(\X, \Y) \, |f(\Y)| \,\d \mu_\M (\Y)  + \int_{\{ \Y; \, d > 1 \}} \sup_{t>0} p_t^\M(\X, \Y) \, |f(\Y)| \,\d\mu_\M (\Y)\\
&:= \HH_{\loc} (f)(\X) +  \HH_\infty (f)(\X).
\end{align*}

The local part $\HH_{\loc}$ is easier to handle. Combining the local heat kernel estimate \eqref{local on W} and $(ND_{\loc})$, $\HH_{\loc}$ can be dominated by a constant multiple of the local centered Hardy-Littlewood maximal operator $\MM_{\loc}$, which is defined by
$$\MM_{\loc}(f)(\X) = \sup_{0< R \leq 1} \frac{1}{V_\M (\X, R)} \int_{B(\X, R)} |f(\Y)|\, \d \mu_\M(\Y), $$
whose weak $(1, 1)$ property follows from ($ND_{\loc}$). Hence we need only consider the part at infinity $\HH_\infty$.

\begin{theo} \label{t5}
$\HH_\infty $ defined on $\M = \H^n \times M$ is of weak type $(1, 1)$, where $M$ is as described in Theorem \ref{nt2}.
\end{theo}

\begin{proof} It suffices to show that the integral operators with the kernels $\sup_{t>0} p_t^\M(\X, \Y) \, \chi_{S_j}$ ($j = 1, 2$) are of weak type $(1, 1)$.

We start with the kernel $\sup_{t>0} p_t^\M(\X, \Y) \, \chi_{S_1}$. By \eqref{hk}, we obtain for $\X = (\u, \x), \Y = (\v, \y) \in \H^n \times M$,
\begin{align*}
\sup_{t>0} p_t^\M(\X, \Y)
%%%&
\leq \sup_{t>0} \sqrt{t} \, p_t^{\H^n}(\u, \v) \cdot \sup_{t>0} \frac{1}{\sqrt{t}} \, p_t^{M}(\x, \y).
\end{align*}
It deduced from \eqref{nBn3}, \eqref{nBn4}, \eqref{nZn2} and \eqref{nBn5} that
$$
\sup_{\Y \in \M} \int_{\M} \sup_{t>0} p_t^\M (\X, \Y) \, \chi_{S_1}(\X, \Y) \, \d \mu_\M(\X) < +\infty.
$$
Therefore the integral operator with the kernel $\sup_{t>0} p_t^\M(\X, \Y) \, \chi_{S_1}$ is bounded on $L^1$.

Then we turn to the analysis of $\sup_{t>0} p_t^\M(\X, \Y) \, \chi_{S_2}$.
Taking an $\varepsilon \in (0, \frac{1}{64})$, by \eqref{hk}, we can write for $r > 1$
\begin{align*}
\sup_{t>0} p_t^\M(\X, \Y)  \leq \sup_{t>0} \sqrt{t} \, p_t^{\H^n}(\u, \v)  \, \e^{\varepsilon \frac{r}{t}} \cdot \sup_{t>0} \frac{1}{\sqrt{t}} \, p_t^{M}(\x, \y) \, \e^{-\varepsilon \frac{r}{t}}.
\end{align*}
Combining this with \eqref{nBn6}, \eqref{nBn7} and \eqref{nZn4}, it follows from Lemmas \ref{weak type on hn} and \ref{weak type 11} that the integral operator with the kernel $\sup_{t>0} p_t^\M(\X, \Y) \, \chi_{S_2}$ is of weak type $(1, 1)$. This finishes the proof.
\end{proof}

\begin{remark}
Recall that the first-order heat maximal operator on $\M$ is defined by
\begin{align*}
\HH_1 (f)(\X) = \sup_{t>0} \bigg| t \, \frac{\partial}{\partial t} \, \e^{t \Delta_\M}f(\X) \bigg|, \quad \forall \, f \in L^p (\M), \, 1 \leq p < + \infty.
\end{align*}
It follows from \cite[Theorem 4]{D97} that $|t \, \frac{\partial}{\partial t} \, p_t^M (x, y)|$ also satisfies the upper bound of type $(GUE_{\theta})$. Using the time derivative estimates of the heat kernel on real hyperbolic spaces (see \cite[Section 5]{LS17}),  a similar argument allows us to obtain the weak $(1, 1)$ property of $\HH_1$.
\end{remark}

\medskip

\renewcommand{\theequation}{\thesection.\arabic{equation}}
\section{Some Generalizations}
\setcounter{equation}{0} \label{S6}

Note that our method works without major difficulties for a large class of direct product manifolds. In this section, we discuss some generalizations.

Consider the direct product manifold $\E \times M$, where $M$ is the same as in Theorem \ref{nt1} or \ref{nt2}, and $\E$ is a connected complete non-compact weighted Riemannian manifold whose volume growth at infinity is at most exponential. We also denote by $r$ the geodesic distance on $\E$. Given $\sigma \in [0, \frac12)$, $\tau \in (0, \frac12]$ and $r_0, \varepsilon > 0$, we define integral operators $\K_i$ ($ = \K_{i, \loc}+\K_{i, \infty}$) $(i = 1, 2)$ on $\E$, with the following kernels respectively,
\begin{align}
K_1(u, v) &= K_{1, \loc}(u, v) + K_{1, \infty}(u, v) \nonumber \\
& := \int_{0}^{\infty} |\nabla_\E p_t^\E (u, v) | \, \frac{\d t}{t^\sigma} \, \chi_{\{ r \leq r_0 \}}  \, + \, \int_{0}^{\infty}  |\nabla_\E p_t^\E(u, v) | \, \e^{\varepsilon \frac{r}{t}} \, \frac{\d t}{t^\sigma} \, \chi_{\{ r > r_0 \}},
\end{align}
and
\begin{align}
K_2(u, v) & = K_{2, \loc}(u, v) + K_{2, \infty}(u, v)    \nonumber  \\
& := \sup_{t > 0} t^\tau \, p_t^\E (u, v) \, \chi_{\{ r \leq r_0 \}} \, + \, \sup_{t > 0}  t^\tau \, p_t^\E (u, v) \, \e^{\varepsilon \frac{r}{t}} \, \chi_{\{ r > r_0 \}}.
\end{align}
We are thus led to the following generalization of Theorems \ref{nt1} and \ref{nt2}.

\begin{theo} \label{general riesz}
Assume that $\E$ satisfies $(ND_{\loc})$ and $(VE)$, whose heat kernel $p_t^\E$ has the small-time upper bound
$$  p_t^\E(u , u)  \leq \frac{C}{V_\E (u, \sqrt{t})}, \quad \forall \, u \in \E, \, \, 0< t \leq 1.$$
Furthermore, suppose that there exist $\sigma \in [0, \frac12)$, $\tau \in (0, \frac12]$ and $r_0, \varepsilon > 0$, such that the integral operators $\K_i$ $(i = 1, 2)$ are of weak type $(1, 1)$. Then the Riesz transform defined on $\E \times M$ is of weak type $(1, 1)$, where $M$ is as described in Theorem \ref{nt1}.
\end{theo}

\begin{theo} \label{general heat maximal}
Assume that $\E$ satisfies $(ND_{\loc})$ and $(VE)$, whose heat kernel $p_t^\E$  has the small-time upper bound
$$  p_t^\E(u , u)  \leq \frac{C}{V_\E(u, \sqrt{t})}, \quad \forall \, u \in \E, \, \, 0< t \leq 1. $$
Furthermore, suppose that there exist $\tau \in (0, \frac12]$ and $r_0, \varepsilon > 0$, such that the integral operator $\K_2$ is of weak type $(1, 1)$. Then the heat maximal operator defined on $\E \times M$ is of weak type $(1, 1)$, where $M$ is as described in Theorem \ref{nt2}.
\end{theo}

The proofs of Theorems \ref{nt1} and \ref{nt2} are organized in such a way that they could be applied to this more general setting %%%immediately.
 via a slight modification. So we omit the proofs of Theorems \ref{general riesz} and \ref{general heat maximal}. It is worth noting that the assumptions of Theorem \ref{general riesz} actually imply the weak type $(1, 1)$ estimate of the Riesz transform defined on $\E$. In fact, it follows from spectral analysis that the weak type $(1, 1)$ estimate of $\K_{1, \infty}$ can deduce that of $\nabla_\E(-\Delta_\E)^{-s}\,\chi_{\{r > r_0 \}}$ for $s \in (0,  1-\sigma]$. While by Theorem \ref{Rloc}, the assumptions on $\E$ imply the weak type $(1, 1)$ estimate of $\nabla_\E(-\Delta_\E)^{-1/2} \, \chi_{\{r \leq r_0 \}}$. Similarly, the weak type $(1, 1)$ of the heat maximal operator defined on $\E$ follows.

\begin{remark}
(1) It is worth pointing out that the factor $\e^{\varepsilon \frac{r}{t}}$ introduced in the definition of $\K_{i, \infty}$$(i = 1, 2)$ is not unique. For example, we can replace it by $(1 + \frac{r}{t})^\varsigma$ with $\varsigma \geq \frac{\upsilon + 1}{2}$, where $\upsilon > 0$ is the constant in \eqref{nd}, with a slight modification of the proof. The advantage of using $\e^{\varepsilon \frac{r}{t}}$ lies in the fact that it only depends on properties of $\E$.

(2)  Theorems \ref{general riesz} and \ref{general heat maximal} remain valid if $M$ is a Lie group with polynomial growth.

\end{remark}

\subsection{Sufficient conditions for $L^1$ boundedness of $\K_{i, \loc}$} \label{S6.1}

Since $\E$ satisfies $(ND_{\loc})$, even more is expected to be said about the local operators $\K_{i, \loc}$ $(i = 1, 2)$. In fact, they will be bounded on $L^1(\E)$ under some additional assumptions, particularly for $\E$ with a positive bottom of the spectrum, such as $\H^n$. Before going further, we first recall that the estimates established in Section \ref{SS32}, particularly \eqref{truncated estimate} and \eqref{Local Estimate}, remain valid on $\E$. Hence they will be used directly.

We begin by analysing $\K_{1, \loc}$.
\begin{align*}
\int_\E K_{1, \loc}(\u, \v)  \, \d \mu_\E(\u) & = \int_{0}^{\infty} \! \bigg( \int_{\{\u; \, r \leq r_0 \}} |\nabla_\E p_t^\E(\u, \v) | \, \d \mu_\E(\u) \bigg) \, \frac{\d t}{t^\sigma}\\
& = \left( \int_{0}^{1} + \int_{1}^{\infty} \right) \! \bigg( \int_{\{\u; \, r \leq r_0 \}} |\nabla_\E p_t^\E (\u, \v) | \, \d \mu_\E (\u) \! \bigg) \, \frac{\d t}{t^\sigma}.
\end{align*}
By \eqref{truncated estimate}, we have
$$  \sup_{\v \in \E} \int_{0}^{1} \! \bigg( \int_{\{\u; \, r \leq r_0 \}} |\nabla_\E p_t^\E (\u, \v) | \, \d \mu_\E (\u) \bigg) \, \frac{\d t}{t^\sigma} \lesssim \int_{0}^{1} \frac{\d t}{t^{\sigma + 1/2}}  < +\infty.  $$

On the other hand, when $t \geq 1$, by Cauchy-Schwarz inequality, we get
\begin{align} \label{gradient norm}
\int_{\{\u; \, r \leq r_0 \}} |\nabla_\E p_t^\E(u, v) | \, \d \mu_\E (u) & \leq V_\E (v, r_0)^{1/2} \, \bigg( \int_\E |\nabla_\E p_t^\E (u, v) |^2 \, \d \mu_\E(u) \bigg)^{1/2}.
\end{align}
Using integration by parts, Cauchy-Schwarz inequality implies that
\begin{align} \label{analytic}
\int_{\E} |\nabla_\E \e^{t \Delta_\E/2} g(\u)|^2 \, \d \mu_\E(\u)  &= -\int_{\E} g(\u) \, (\Delta_\E \e^{t\Delta_\E} g)(\u) \, \d \mu_\E(\u)\nonumber\\
& \leq \| g \|_2 \,  \| \Delta_\E \e^{t\Delta_\E} g \|_2 \leq t^{-1} \| g \|_2^2, \qquad \forall \, t > 0, \, g \in C_0^\infty(\E),
\end{align}
where we have used in the last inequality the analyticity of $(\e^{t\Delta_\E})_{t > 0}$ on $L^2$. It follows from the semigroup property that $p_t^\E(\u, \v) = \e^{{t\Delta_\E}/2}(p_{t/2}^\E(\cdot, \v))(\u)$. Therefore combining \eqref{gradient norm} and \eqref{analytic} gives
$$
\int_{\{\u; \, r \leq r_0 \}} |\nabla_\E p_t^\E(\u, \v) | \, \d \mu_\E (u) \leq \frac{V_\E (\v, r_0)^{1/2} \, \| p_{t/2}^\E(\cdot, \v) \|_2}{t^{1/2}} = \frac{\sqrt{V_\E (v, r_0) \, p_t^\E(\v, \v)}}{t^{1/2}}.
$$

Finally, we conclude that if
\begin{align} \label{criterion1}
\sup_{\v \in \E} \int_{1}^{\infty} \frac{\sqrt{V_\E (v, r_0) \, p_t^\E(\v, \v) }}{t^{\sigma + 1/2}} \, \d t < +\infty
\end{align}
holds, then $\K_{1, \loc}$ is bounded on $L^1(\E)$.

Now, we turn to the integral operator $\K_{2, \loc}$. It follows from \eqref{Local Estimate} that there exists a constant $C > 0$, such that for all $u, v \in \E$ with $r = r(u, v)  \le r_0$,
\begin{align} \label{K2 loc}
K_{2 ,\loc}(u, v) & \leq \sup_{0< t \le 1} t^\tau \, p_t^\E (u, v) \, + \, \sup_{t > 1} t^\tau \, p_t^\E (u, v)  \nonumber\\
& \lesssim \sup_{0 < t \leq 1} \frac{t^\tau}{V_\E (u, \sqrt{t})} \e^{-\frac{r^2}{Ct}} \, + \, \sup_{t > 1} t^\tau \, p_t^\E (u, v).
\end{align}

First, we claim that there exist positive constants $\upsilon_1$ and $\upsilon_2$ such that
\begin{align} \label{polynomial}
\frac{V_\E (u, R_1)}{V_\E (u, R_2)}  \lesssim_{r_0} \bigg( \frac{R_1}{R_2} \bigg)^{\upsilon_1} \! \! \! + \bigg( \frac{R_1}{R_2} \bigg)^{\upsilon_2}, \qquad \forall \, u \in \E, \, R_1 \in (0, r_0], \, R_2 \in (0, 1].
\end{align}
Indeed, when $R_1 \geq R_2$, by iterating $(ND_{\loc})$, there exists $\upsilon_1 > 0$ such that
$$
\frac{V_\E (u, R_1)}{V_\E (u, R_2)}  \lesssim_{r_0} \bigg( \frac{R_1}{R_2} \bigg)^{\upsilon_1}, \qquad  0 < R_1 \le r_0.
$$
When $R_1 < R_2$, the so-called (local) reversed volume doubling property (see  \cite[page 412]{G09}), which is a direct consequence of $(ND_{\loc})$ and the connectedness of $\E$, implies that there exists $\upsilon_2 > 0$ such that
$$
\frac{V_\E (u, R_1)}{V_\E (u, R_2)} \lesssim_{r_0} \bigg( \frac{R_1}{R_2} \bigg)^{\upsilon_2}.
$$
In both cases, \eqref{polynomial} is true.

Then applying \eqref{polynomial} to the first term in \eqref{K2 loc}, we have
\begin{align}
\sup_{0 < t \leq 1} \frac{t^\tau}{V_\E (u, \sqrt{t})} \e^{-\frac{r^2}{C t}} \lesssim_{r_0} \sup_{0 < t \leq 1} \bigg(\bigg( \frac{r}{\sqrt{t}} \bigg)^{\upsilon_1-2\tau} \! \! \! + \bigg( \frac{r}{\sqrt{t}} \bigg)^{\upsilon_2-2\tau}  \bigg) \e^{-\frac{r^2}{Ct}} \cdot \frac{r^{2\tau}}{V_\E (u, r)}.
\end{align}
By the same argument used in \eqref{nBn5}, we deduce that for any $\tau > 0$,
$$
\sup_{\v \in \E} \int_{\{\u; \, r \leq r_0 \} } \frac{r^{2\tau}}{V_\E (u, r)} \, \d \mu_\E(u) < +\infty.
$$
Set $\upsilon' = \min \{\upsilon_1, \upsilon_2\}$. Therefore if $\tau \in (0, \frac{\upsilon'}{2}]$, the integral operator with the kernel $\sup_{0 < t \le 1} t^\tau \, p_t^\E \chi_{\{ r \leq r_0 \}}$ is bounded on $L^1(\E)$.

In conclusion, $\K_{2, \loc}$ is bounded on $L^1(\E)$ whenever there exists $\tau \in (0, \frac{\upsilon'}{2}]$ such that
\begin{align} \label{criterion2}
\sup_{\v \in \E} \int_{\{\u; \, r \le r_0 \}} \sup_{t > 1} t^\tau \, p_t^\E (u, v) \, \d \mu_\E (u) < +\infty.
\end{align}

In particular, if $\E$ has a positive bottom of the spectrum $\lambda_1$, then by \eqref{Local Estimate} and $(ND_{\loc})$, the conditions \eqref{criterion1} and \eqref{criterion2} can be verified easily.

\subsection{More specific examples}

To end this section, we present more specific examples of $\E$. We start with the following simple but instructive model. Given $\alpha > 0$, for non-negative numbers $a_1$ and $a_2$ that are not all zero, the function
$$ h(u) = a_1 \e^{\alpha u} + a_2 \e^{-\alpha u} $$
is a positive $\alpha^2$-harmonic function defined on $\R$. Moreover, all the positive eigenfunctions of $\Delta_{\R}$ corresponding to a positive eigenvalue take this form. Let $\mu$ be the measure on $\R$ defined by $\d \mu = h^2 \d u$, where $\d u$ is the Lebesgue measure. We consider the weighted Riemannian manifold $\E = (\R, \d \mu)$, equipped with the canonical Euclidean metric. It is easily seen that for all $u \in \E$,
\begin{equation*}
 V_\E (u, R)  \, \sim \,
 \begin {cases}
 h^2(u) R  & \text{if \, $0 < R \leq 1$,} \\
  h^2(u) \e^{2 \alpha R}   & \text{if \, $R>1$.} \\
   \end {cases}
\end{equation*}

The (weighted) Laplace operator on $\E$ is as follows:
$$ \Delta_\E  = \frac{\d^2}{\d u^2} + 2 \frac{h'}{h} \frac{\d}{\d u} = \frac{\d^2}{\d u^2} + 2\alpha \frac{a_1\e^{2 \alpha u}-a_2}{a_1 \e^{2\alpha u}+ a_2 } \frac{\d}{\d u}.   $$

By Doob's $h$-transform (see \cite{G09}), the heat kernel on $\E$ is explicitly given by
$$ p_t^\E(u ,v)  = \frac{1}{\sqrt{4\pi t}} \exp{\left(-\frac{r^2}{4t} - \alpha^2 t\right)} \frac{1}{h(u)h(v)},$$
and a direct computation gives rise to
$$ |\nabla_\E p_t^\E (u, v)  | \lesssim  t^{-\frac32} \left( r + t \right) \exp{\left(-\frac{r^2}{4t} - \alpha^2 t\right)} \frac{1}{h(u)h(v)}, $$
where $r = r(u, v) = |u - v|$.

To apply Theorems \ref{general riesz} and \ref{general heat maximal}, it only remains to verify the weak type $(1, 1)$ of $\K_i$ $(i = 1, 2)$, since other assumptions on $\E$ are satisfied trivially. To do this, take $\sigma = 0$, $\tau = \frac12$, $r_0 = 1$ and $\varepsilon \in (0, \frac18)$ in the expressions of $K_i$.  Remark that the local operators $\K_{i, \loc}$ $(i = 1, 2)$ are bounded on $L^1$ because $\lambda_1 = \alpha^2$ in this setting.

On the other hand, we have
\begin{align*}
K_{1, \infty}(u, v) \lesssim \int_{0}^{\infty} t^{-\frac32} \left( r + t \right) \exp{\left(-\frac{r^2}{4t} - \alpha^2 t + \frac{\varepsilon r}{t}\right)} \, \d t \cdot \frac{1}{h(u)h(v)} \lesssim \frac{\e^{-\alpha r}}{h(u)h(v)},
\end{align*}
and
\begin{align*}
K_{2, \infty}(u, v) \lesssim \sup_{t>0}\, \exp{\left(-\frac{r^2}{4t} - \alpha^2 t + \frac{\varepsilon r}{t}\right)} \cdot \frac{1}{h(u)h(v)} \lesssim \frac{\e^{-\alpha r}}{h(u)h(v)}.
\end{align*}
Notice that $\frac{\e^{-\alpha r}}{h(u)h(v)} \leq \frac{1}{h^2(u)} \in L^{1, \infty}(\E)$, the weak type $(1, 1)$ of $\K_i$ follows. Therefore $(\R, \d \mu)$ is an example of $\E$ in our theorems. In particular, it is known that the Euclidean space $\R^n$ endowed with the drifted Laplacian can be viewed as a direct product manifold in essence. Then if we take $\E = (\R, \e^{2u}\d u)$ and $M = (\R^{n - 1}, \d x)$, applying Theorems \ref{general riesz} and \ref{general heat maximal} gives an alternative proof of Theorems 1 and 2 in \cite{LSW16}.

\medskip

We emphasize that the assumptions on $\E$ could be verified on some manifolds with a positive bottom of the spectrum. As illustrated in Section \ref{S6.1}, the main difficulty in this context is the weak $(1, 1)$ of $\K_{i, \infty}$ $(i = 1,2)$, which usually follows from the analogues of Lemma \ref{weak type on hn}, involving sharp estimates of heat kernel and its gradient. According to the known results on the weak type $(1, 1)$, Theorems \ref{general riesz} and \ref{general heat maximal} are valid when $\E$ is taken as the following manifolds mentioned in Section \ref{SS12}: non-compact symmetric spaces, harmonic $AN$ groups, the Laplacian with drift on real hyperbolic spaces, as well as the sub-Laplacian with drift on Heisenberg groups. As for specific properties of guaranteeing the validity of our assumptions, we refer the reader to \cite{A92, S81}, \cite{ADY96}, \cite{LS17} and \cite{LS20}. The details are left to the reader.

\medskip

\renewcommand{\theequation}{\thesection.\arabic{equation}}
\section{Appendix: two counterexamples for $p > 2$}
\setcounter{equation}{0} \label{S7}
As already mentioned in our introduction, the $L^p$ boundedness ($p > 2$) of the Riesz transform on $\M = \H^n \times M$ is a quite subtle problem. In general, $(ND)$, $(GUE_{\theta})$ together with $(WGE_{q, \theta})$ of $M$ is not sufficient. To illustrate this, we shall present two typical counterexamples. In this appendix, we always assume $M$ satisfies $(D)$ and $(UE_2)$, i.e. the quasi-distance $\widetilde{d}$ is the geodesic distance $d_M$. Let us recall that $(WGE_{2, 1})$ is a corollary of our assumptions.

\medskip

\noindent{\bf Example \ref{S7}.1 } For $k \ge 2$, let $M = \R^k \# \R^k$ be the connected sum of two copies of $\R^k$, where the Riesz transform is bounded on $L^p$ iff $1 < p \le 2$ when $k = 2$ and $1 < p < k$ when $k \ge 3$ (see \cite{CD99,CCH06}). By its construction, the Ricci curvature $\Ric_M$ equals zero identically outside a compact set, hence is bounded from below. On the other hand,
it is well-known that $\Ric_{\H^n} \equiv -(n - 1)$.
Therefore the Ricci curvature is bounded from below on $\M = \H^n \times M$. Then applying the main theorem of \cite{B87}, we obtain for any given $1 < p < +\infty$,
$$
\| \nabla_\M f \|_p \lesssim \| (- \Delta_\M)^{\frac12} f   \|_p + \| f  \|_p, \quad \forall \, f \in C_0^\infty(\M).
$$
And one easily sees that $\M$ has a positive bottom of the spectrum. By interpolation,
$$
\| f \|_p \lesssim \|  (- \Delta_\M)^{\frac12} f \|_p, \quad \forall \, f \in C_0^\infty(\M).
$$
Combining with the above inequalities, the $L^p$ boundedness of the Riesz transform on $\M$ holds for any $1 < p < +\infty$. It is a bit surprising since taking direct product enlarges the range of $p$ for the $L^p$ boundedness of the Riesz transform. In addition, this argument does not use $(D)$ and $(UE_2)$, so the same phenomenon happens when $M$ is taken as the connected sums considered in \cite{HS19}.

\medskip

\noindent{\bf Example \ref{S7}.2 } Let $M$ be a $k$-dimensional conical manifold with a compact boundaryless basis. The Riesz transform on $M$ has been studied in \cite{Li99}. Let $\lambda_1$ be the smallest strictly positive eigenvalue of the basis. Then the $L^p$ boundedness of the Riesz transform is valid iff $1 < p < p_0$, where
\begin{equation*}
 p_0  \, = \,
 \begin {cases}
 k \bigg( \frac{k}{2} - \sqrt{\big( \frac{k-2}{2} \big)^2 + \lambda_1} \bigg)^{-1}  & \text{if \, $\lambda_1 < k-1$,} \\
 \, \, \, +\infty   & \text{if \, $\lambda_1 \ge k - 1$.} \\
   \end {cases}
\end{equation*}
We shall only focus on the case $\lambda_1 < k - 1$. For any fixed $p \ge p_0$, suppose that the Riesz transform on $\M = \H^n \times M$ is bounded on $L^p$, i.e.
\begin{equation} \label{assumption}
\| \nabla_\M F   \|_p \lesssim \| (-\Delta_\M)^{\frac12} F \|_p, \quad \forall \, F \in D((-\Delta_\M)^{\frac12}).
\end{equation}

By the results of \cite[p. 239]{CL04}, there exists $g \in C_0^\infty (M)$ and $t_0 > 0$, such that
\begin{align*}
\| \nabla_M \e^{t_0 \Delta_M} g \|_p = +\infty.
\end{align*}
Now we choose an $f \in C_0^\infty (\H^n)$ that is not identically 0, and set
\begin{align*}
G(\u, \x) & := \e^{t_0 \Delta_{\H^n}}f(\u) \cdot \e^{t_0 \Delta_{M}} g(\x) \\
& = \e^{t_0 \Delta_\M}(f \otimes g) \in D((-\Delta_\M)^{\frac12}).
\end{align*}
Then it is easily checked that
\begin{align*}
\| \nabla_\M G \|_p &\sim \| \nabla_{\H^n} \e^{t_0 \Delta_{\H^n}} f \|_p \cdot \|  \e^{t_0 \Delta_{M}} g \|_p + \| \e^{t_0 \Delta_{\H^n}} f \|_p \cdot \|  \nabla_M \e^{t_0 \Delta_{M}} g \|_p \\
& = +\infty.
\end{align*}
On the other hand, by the analyticity of $(\e^{t\Delta_\M})_{t > 0}$ on $L^p$, we obtain
\begin{align*}
\| (-\Delta_\M)^{\frac12} G \|_p  &= \| (-\Delta_\M)^{\frac12} \e^{t_0 \Delta_\M}(f \otimes g) \|_p \\
& \le t_0^{-\frac12} \, \|  f \otimes g \|_p < +\infty,
\end{align*}
which contradicts \eqref{assumption}. So when $p_0 \leq p < +\infty$, the Riesz transform on $\M$ is not $L^p$ bounded.

\medskip

As for the $L^p$ boundedness ($p > 2$) of the Riesz transform on $\M = \H^n \times M$ for a general $M$, we can apply the criterions developed in \cite{ACDH03}. Indeed, by \cite[Theorem 5.5.3]{S02}, the conjunction of $(ND_{\loc})$ and the local Poincar\'e inequality is equivalent to the two-sided Gaussian heat kernel estimate for small time. Accordingly, \cite[Theorem 1.8]{ACDH03} asserts the following result:

Assume $M$ is a connected stochastically complete non-compact Riemannian manifold satisfying $(D)$, whose heat kernel has the two-sided Gaussian estimate for small time: there exist constants $C$ and $c$, such that for all $x, y \in M$ and $0 < t < 1$,
$$
\frac{c}{V_M(x, \sqrt{t})} \exp{\left( - C \, \frac{d_M(x, \, y)^2}{t} \right)} \leq p_t^M(x, y) \leq \frac{C}{V_M(x, \sqrt{t})} \exp{\left( - c \, \frac{d_M(x, \, y)^2}{t} \right)}.
$$
If for some $p_0 \in (2, +\infty]$ and $\alpha \ge 0$, it holds for all $t > 0$ that
\begin{align*}
\| |\nabla_M \e^{t \Delta_M}| \|_{p_0 \to p_0} \lesssim \frac{\e^{\alpha t}}{\sqrt{t}},
\end{align*}
then the Riesz transform on $\M$ is bounded on $L^p$ for all $p \in [2, p_0)$.

\medskip

\section*{Acknowledgement}
This work is partially supported by NSF of China (Grants No.11625102 and No. 12271102). The second author is also partially supported by the National Key R\&D Program of China (No. 2022YFA1006000, 2020YFA0712900) and NNSFC (11921001). The authors would like to thank the anonymous referees for their many useful suggestions and valuable remarks which improve the writing of the paper.

\bigskip

\mbox{}\\
Hong-Quan Li\\
School of Mathematical Sciences  \\
Fudan University \\
220 Handan Road  \\
Shanghai 200433  \\
People's Republic of China \\
E-Mail: \url{hongquan_li@fudan.edu.cn} \\

\mbox{}\\
Jie-Xiang Zhu\\
Center for Applied Mathematics \\
Tianjin University \\
Tianjin 300072 \\
People's Repubilc of China \\
E-Mail: \url{15110840006@fudan.edu.cn}

\end{document}